\documentclass[leqno,10pt]{amsart}

\usepackage{amsmath,amssymb,mathrsfs,mathtools,stmaryrd,comment} 
\usepackage[
colorlinks,linkcolor=blue,citecolor=blue,urlcolor=red]{hyperref}
\usepackage{times}
\usepackage{gensymb}
\usepackage[all]{xy}

\theoremstyle{plain}
    \newtheorem{thm}{Theorem}[section]
    \newtheorem{ppn}[thm]{Proposition}
    \newtheorem{lem}[thm]{Lemma}
    \newtheorem{cor}[thm]{Corollary}
\theoremstyle{definition}
    \newtheorem{dfn}[thm]{Definition}
    
\theoremstyle{remark}
    \newtheorem{rmk}[thm]{Remark}
    \newtheorem*{rmk*}{Remark}

\numberwithin{equation}{section}

\allowdisplaybreaks

\def\Ker{\operatorname{Ker}}\def\Hom{\operatorname{Hom}}
\def\C{\mathbb{C}}\def\Q{\mathbb{Q}}\def\R{\mathbb{R}}\def\Z{\mathbb{Z}}\def\F{\mathbb{F}}
\def\ol#1{\overline{#1}}\def\wt#1{\widetilde{#1}}

\def\ot{\otimes}
\def\a{\alpha}\def\b{\beta}\def\g{\gamma}\def\pd{\partial}
\def\l{\lambda}\def\k{\kappa}\def\z{\zeta}\def\x{\xi}\def\y{\eta}
\def\o{\omega}

\def\angle#1{{\langle #1 \rangle}}

\def\G{\Gamma}

\def\s{\sigma}\def\t{\tau}
\def\Im{\operatorname{Im}}

\def\D{\Delta}

\def\r{\rho}

\newcommand\sO{\mathscr{O}}

\newcommand\be{\mathbf{e}}

\newcommand{\can}{\mathrm{can}}

\newcommand{\End}{\mathrm{End}}
\newcommand{\Gal}{\mathrm{Gal}}

\newcommand{\Red}{\mathrm{Red}}

\newcommand{\A}{{\mathbb{A}}}

\newcommand{\bmm}{\boldsymbol\mu}

\renewcommand{\H}{\mathbb{H}}
\renewcommand{\P}{\mathbb{P}}

\newcommand\n{\nu}\newcommand\m{\mu}
\newcommand\vp{\varphi}

\newcommand\0{\circ}

\newcommand{\Fr}{\mathrm{Fr}}

\newcommand{\Tr}{\mathrm{Tr}}

\newcommand\vt{\vartheta}
\renewcommand\th{\theta}

\newcommand\AGM{\mathrm{AGM}}
\newcommand\cn{\mathrm{cn}}
\newcommand\dn{\mathrm{dn}}
\newcommand\sn{\mathrm{sn}}

\newcommand{\pmat}[1]{\begin{pmatrix}#1\end{pmatrix}}
\newcommand{\bmat}[1]{\begin{bmatrix}#1\end{bmatrix}}

\newcommand{\Th}{\Theta}

\begin{document}

\title{Cubic arithmetic-geometric mean and isogenies of Hesse elliptic curves}
\author{Noriyuki Otsubo}
\address{Department of Mathematics and Informatics, Chiba University, Inage, Chiba, 263-8522 Japan}
\email{otsubo@math.s.chiba-u.ac.jp}

\begin{abstract}
We construct degree-three isogenies of Hesse elliptic curves over an arbitrary field of characteristic different from three, which correspond to a step of the cubic arithmetic-geometric mean. 
As an application, we describe the Serre--Tate canonical lift of an ordinary elliptic curve in characteristic three by the $3$-adic cubic arithmetic-geometric mean. 
We also discuss the relations of the Hesse curves with 
complex and finite hypergeometric functions. 
\end{abstract}

\date{\today.}

\maketitle

\section{Introduction} 

For an integer $N \ge 2$, the step of the degree $N$ arithmetic-geometric mean is given by
\[(a,b)\mapsto (a',b'); \quad a'=\frac{a+(N-1)b}{N}, \quad c'=\frac{a-b}{N}, \quad a'^N=b'^N+c'^N.\]
Here we assume that the $N$-th root defining $b'$  is canonically or explicitly given. 
When $N=2$, this is the classical AGM step $(a',b')=(\frac{a+b}{2},\sqrt{ab})$. 
When $N=3$, it becomes 
\[(a',b')=\left(\frac{a+2b}{3}, \sqrt[3]{b\frac{a^2+ab+b^2}{3}}\right).\]
Iterating this step starting with $(a_0,b_0)=(a,b)$, we obtain the AGM sequence $(a_n,b_n)$. 
If $a$ and $b$ are real numbers satisfying $0<b \le a$, then $a_n$ and $b_n$ converge with order $N$ to the same limit, written $\AGM_N(a,b)$ and called  the {\it degree-$N$ AGM} of $a$ and $b$. 

It is classically known by Gauss, Jacobi and others that the quadratic AGM is deeply related to theta functions, the hypergeometric function $F(\frac{1}{2},\frac{1}{2};1;\l)$ and the Legendre elliptic curve
\[L(\l) : y^2=x(1-x)(1-\l x) \quad (\l \ne 0,1)\]
 (see Sections \ref{s3} and \ref{s5}). 
 Over $\C$, we have a uniformization
 \[\Psi_\t \colon \C/\angle{1,\t} \xrightarrow\simeq L(\k^2),\] 
 where $\Im\t>0$, $q=e^{\pi i \t}$ and $\k=\th_3^2(q)/\th_4^2(q)$. Here, $\th_k(q)$ are Jacobi's theta constants and 
 \[(a_n,b_n)=\bigl(\th_3^2(q^{2^n}), \th_4^2(q^{2^n})\bigr)\]
form a quadratic AGM sequence (see Section \ref{s2.1}). 
Therefore, there exists an isogeny of degree $2$ 
\[g\colon L(\k^2) \to L(\k'^2), \quad \k'\coloneqq \frac{(\k+1)^2}{4\k}\]
which corresponds to the map $\C/\angle{1,\t} \to \C/\angle{1, 2\t}$; $z \mapsto 2z$ (see Section \ref{s3.2}). 
 
When $N=3$, Borwein--Borwein \cite{bb} defined cubic theta constants $a(q)$, $b(q)$, $c(q)$ and showed that 
\[(a_n,b_n) = \bigl(a(q^{3^n}),b(q^{3^n})\bigr)\]
is a cubic AGM sequence, with $c_n=c(q^{3^n})$ (see Section \ref{s2.3}). 
Recall that the Legendre curve is the universal elliptic curve over the modular curve $X(\G(2))$, 
while the universal elliptic curve over $X(\G(3))$ is the Hesse elliptic curve defined by the homogeneous equation 
\[H(\m) : X^3+Y^3+Z^3 -3 \m XYZ=0 \quad (\m^3\ne 1).\] 
The first aim of this paper is to give a cubic analogue of the complex uniformization $\Psi_\t$ and the isogeny $g$ for the Hesse curve. 
In fact, we give two uniformizations 
\begin{align*}&\Phi_\t\colon \C/\angle{1,\t} \xrightarrow\simeq H(\r(q)), \quad \r(q)\coloneqq a(q^2)/c(q^2),
 \\& \Psi_\t\colon \C/\angle{1,\t} \xrightarrow\simeq H(\m(q)), \quad \m(q)\coloneqq a(q^\frac{2}{3})/b(q^\frac{2}{3}) 
\end{align*} 
(Theorems \ref{okada1}, \ref{t3.11}). 
The former is more symmetric, but the latter is more directly related to the cubic AGM because of the definition of $\m(q)$. 
We also give explicit formulas for the degree-three isogenies 
\[f\colon H(\r(q)) \to H(\r(q^3)), \quad g\colon H(\m(q)) \to H(\m(q^3))\] 
corresponding to the map 
$\C/\angle{1,\t} \to\C/\angle{1,3\t}$; $z \mapsto 3z$ (Theorems \ref{okada2}, \ref{t4.8}). 
Theorems \ref{okada1} and \ref{okada2}  are established in Okada's master's thesis \cite{okada}. 
These four theorems follow from algebraic relations among cubic theta functions, and all of these are provided by Liu \cite{liu} (see Section \ref{s2.3}). 

Once such isogenies are found over $\C$, it is not difficult to verify that the same formulas work over any field. 
Let $K$ be a field of characteristic different from $3$, and $\bmm_3$ denote the set of cubic roots of unity in an algebraic closure of $K$. The formula for $g$ is as follows (see Theorem \ref{alg-isogeny} for $f$). 

\begin{thm}[Theorem \ref{alg-isogeny2}]\label{t1.1} 
Suppose that $\m, \m' \in K-\bmm_3$ satisfy 
\[\m'^3=\left(\frac{\m+2}{3}\right)^3 \left(\frac{\m^2+\m+1}{3}\right)^{-1},\]
and put 
\[\n= \begin{cases} \frac{\m+2}{3\m'} & (\m'\ne 0), \\ 1& (\m'=0). \end{cases}\]
Then, there exists an isogeny $g\colon H(\m) \to H(\m')$ of degree $3$ defined by 
\[[X:Y:Z] \mapsto [3 \n XYZ: X^2Y+Y^2Z+Z^2X:XY^2+YZ^2+ZX^2],\]
and its kernel is given by $X+Y+Z=0$. 
\end{thm}

We remark that Perez Broon et al. \cite{pbetal} constructed several $3$-isogenies between twisted Hesse elliptic curves with two parameters. Very recently, as an application of the aforementioned result, Hashimoto--Nuida \cite[Theorem 4.2]{hn} independently constructed the same isogeny as our $g$, although they work over finite fields. 

Now, let $p$ be a prime number and $K$ be a finite extension of $\Q_p$. 
Let $|\ |$ be the absolute value on $K$, normalized by $|p|=p^{-1}$. 
If $a, b \in K^*$ and $\m=a/b$ is sufficiently close to $1$, 
then the quadratic (resp. cubic) AGM sequence is canonically defined, and $a_n$ and $b_n$ converge quadratically (resp. cubically) to the limit $\AGM_2(a,b)$ (resp. $\AGM_3(a,b)$) in $K$ (see Section \ref{s6.1}). 

The $p$-adic quadratic AGM was first studied by Henniart--Mestre \cite{hm},    
who used this to compute the $p$-adic uniformization (Tate's parametrization) of an elliptic curve over $K$ with multiplicative reduction. 
Moreover, Mestre \cite{mestre}, Satoh \cite{satoh} and others used the $2$-adic quadratic AGM to compute the Serre--Tate canonical lift $E^\can$ of an ordinary elliptic curve $E$ over a finite field of characteristic $2$.  
Since the Frobenius endomorphism of $E$ lifts to $E^\can$, this gives an efficient method for  computing   
the Frobenius trace of $E$, or equivalently, the number of rational points on $E$, which is of significant interest in cryptography. 
The last result is generalized and well-presented by Kinjo--Miyasaka \cite{km}. 
We also remark that the quadratic isogeny sequence of Legendre curves over a finite field was studied by Griffin--Ono--Saikia--Tsai \cite{gost} (see \cite{hn} for the cubic analogue).  

Our main result in the $p$-adic setting is a cubic analogue of \cite[Theorem 1.1]{km}. 
Suppose that $p=3$ and $\m=a/b$ ($a, b \in K^*$) satisfies $|\m-1| = |3|^2$. 
Then the cubic AGM sequence $(a_n, b_n)$ and hence $\m_n=a_n/b_n$ are well-defined, but 
$\m_n$ does not necessarily converge (and it never converges to $1$). 
In this case, the Hesse curve $H(\m)$ has ordinary good reduction (see Section \ref{s6.2}). 

\begin{thm}[Theorem \ref{t7.1}]\label{t1.2}
Let $d$ be the residual degree of $K$. 
Then the limit 
\[\m_\infty \coloneqq \lim_{n\to\infty} \m_{dn}\]
exists, and $H(\m_\infty)$ is isomorphic over $K(\bmm_3)$ to the Serre--Tate canonical lift of the reduction of $H(\m)$. 
\end{thm}

Although it may be less widely known to number theorists, the addition law on $H(\m)$ is extremely simple (see Section \ref{s4.1}), and this fact is one of the reasons why this curve is well studied in cryptography (see for example \cite{fj,jq}). 
The isogenies $f$ and $g$ are also quite simple. 
It would be interesting to apply the results of this paper to computational studies of elliptic curves over finite fields, such as efficient point counting. 

This paper is organized as follows. In Section \ref{s2}, we recall the necessary definitions and properties of Jacobi theta functions, Borwein--Borwein theta constants and cubic theta functions. 
In Section \ref{s3}, we first recall the complex uniformizations of the Jacobi and Legendre elliptic curves,  and give isogenies corresponding to the quadratic AGM step. 
In Section \ref{s4}, after collecting necessary materials on Hesse elliptic curves, we give the uniformizations $\Phi_\t$ and $\Psi_\t$, clarifying the relation between the two, 
and give the isogenies $f$ and $g$ (and their dual isogenies). 
In Section \ref{s5}, we discuss connections between Hesse curves and the complex hypergeometric function $F(\frac{1}{3},\frac{2}{3};1;\m)$ and its analogues over finite fields, which appear respectively  as the complex period and the Frobenius trace. 
In Section \ref{s6}, we first study the convergence of the $p$-adic cubic AGM. 
Then, after studying the reduction of Hesse curves over a $p$-adic field for any $p$, we study the Serre--Tate canonical lift of elliptic curves over a $3$-adic field. 

\section*{Notations}

For $z\in \C$, we write $\be(z)=e^{2\pi i z}$. 
Let $\mathbb{H}=\{\t \in \C \mid \Im \t>0\}$ be the upper half plane. 
Throughout this paper, we write 
\[q=\be(\t/2) \quad (\t \in \mathbb{H}), \quad u=\be(z) \quad (z\in\C).\]
Note that $|q|<1$. 
For any $a \in \R$, we write $u^a =\be(az)$ and $q^a=\be(a\t/2)$. 
For any $\t \in \mathbb{H}$, 
\[\langle 1, \t \rangle \coloneqq \Z+\Z\t\] is a lattice in $\C$. 
We simply write $n \equiv a\ (N)$ instead of $n \equiv a \pmod{N}$. 
Let $\bmm_3$ denote the group of cubic roots of unity in a fixed algebraically closed field of characteristic $\ne 3$, and $\o\in\bmm_3$ be a primitive root.  
We set $\o=\be(1/3)$ when the field is $\C$. 
For an elliptic curve $E$ and a positive integer $n$, let $E[n]$ denote the subgroup of $n$-torsion points. 

\section{Theta functions}\label{s2}

In this section, we prepare necessary materials on theta functions.  
We provide short proofs for several formulas for completeness and the convenience of the reader. 

\subsection{Jacobi theta functions}\label{s2.1}
General references are \cite{tata} and \cite{ww}. 
The standard {\it theta function} is defined by
\[\vt(z,\t) = \sum_{n\in\Z} q^{n^2}u^n .\]
It satisfies
\[\vt(z+m+n\t,\t)=  q^{-n^2}u^{-n} \vt(z,\t) \quad (m,n\in\Z).\]
It is an entire function and has a simple zero at each $z \equiv 1/2+\t/2 \pmod{\angle{1,\t}}$. 
For $a, b \in \C$, the variant of $\vt(z,\t)$ is defined by 
\[\vt_{a,b}(z,\t)= \sum_{n\in\Z} \be\left((n+a)^2\t/2+(n+a)(z+b)\right)=\be(ab)q^{a^2} u^a\vt(z+a\t+b,\t). \]

The {\it Jacobi theta functions} are defined by 
\begin{align*}
& \th_1(z,q)= - \vt_{\frac{1}{2},\frac{1}{2}}(z,\t) = - i \sum_{n\equiv 1 (2)} (-1)^\frac{n-1}{2}  q^\frac{n^2}{4}u^\frac{n}{2}, 
\\& \th_2(z,q)= \vt_{\frac{1}{2},0}(z,\t) =\sum_{n \equiv 1 (2)} q^\frac{n^2}{4}u^\frac{n}{2} , 
\\& \th_3(z,q)=   \vt_{0,0}(z,\t) =\sum_{n\in\Z} q^{n^2} u^n, 
\\& \th_4(z,q)= \vt_{0,\frac{1}{2}}(z,\t) =\sum_{n \in\Z} (-1)^n q^{n^2} u^n. 
\end{align*}
These satisfy the (anti-) symmetry and the quasi-periodicity (we just write $\th_m(z)=\th_m(z,q)$):  
\begin{align}\label{e2.1}
& \th_1(-z)= -\th_1(z), & & \th_1(z+1)=-\th_1(z), & & \th_1(z+\t)=-(qu)^{-1} \th_1(z), 
\\& \th_2(-z)= \th_2(z), & & \th_2(z+1)=-\th_2(z), & & \th_2(z+\t)= (qu)^{-1}\th_2(z), \notag
\\& \th_3(-z)= \th_3(z), & & \th_3(z+1)=\th_3(z), & & \th_3(z+\t)=(qu)^{-1}\th_3(z), \notag
\\& \th_4(-z)= \th_4(z), & & \th_4(z+1)=\th_4(z), & & \th_4(z+\t)=-(qu)^{-1} \th_4(z). \notag
\end{align}
These are mutually related by 
\begin{align}\label{e2.2}
& \th_1(z+1/2)=\th_2(z), & & \th_2(z+1/2)= - \th_1(z), 
\\&  \th_3(z+1/2)=\th_4(z), & &  \th_4(z+1/2)=\th_3(z),\notag
\\&\th_1(z+\t/2)= i q^{-\frac{1}{4}} u^{-\frac{1}{2}} \th_4(z) & & \th_2(z+\t/2)= q^{-\frac{1}{4}}u^{-\frac{1}{2}}\th_3(z), \notag
\\&\th_3(z+\t/2)= q^{-\frac{1}{4}}u^{-\frac{1}{2}}\th_2(z), & & \th_4(z+\t/2)= i q^{-\frac{1}{4}}u^{-\frac{1}{2}}\th_1(z). \notag
\end{align}

Recall the {\it $q$-Pochhammer symbol} \[(x;q)_\infty=\prod_{n=0}^\infty (1-xq^n).\] 
The {\it Dedekind eta function} is defined by 
\[\y(q)=q^\frac{1}{24}(q;q)_\infty. \]
Note that the usual eta function $\y(\t)$ is our $\y(q^2)$ by our convention $q=\be(\t/2)$. 
Then Jacobi's triple product formula (cf. \cite[Theorem 0.1]{cooper}) is rephrased as: 
\begin{align}\label{triple}
& \th_1(z,q)=-i q^\frac{1}{4}(u^\frac{1}{2}-u^{-\frac{1}{2}}) (q^2;q^2)_\infty(uq^2;q^2)_\infty(u^{-1}q^2;q^2)_\infty, 
\\& \th_2(z,q)=q^\frac{1}{4}(u^\frac{1}{2}+u^{-\frac{1}{2}}) (q^2;q^2)_\infty(-uq^2;q^2)_\infty(-u^{-1}q^2;q^2)_\infty, \notag
\\& \th_3(z,q)= (q^2;q^2)_\infty(-uq;q^2)_\infty(-u^{-1}q;q^2)_\infty,  \notag
\\& \th_4(z,q)= (q^2;q^2)_\infty(uq;q^2)_\infty(u^{-1}q;q^2)_\infty. \notag
\end{align}

Define the {\it Jacobi theta constants} by 
\[\th_m(q)\coloneqq \th_m(0,q) \quad (m=1,2,3,4). \]
 Then we have by \eqref{triple} 
\begin{align}\label{const2}
& \th_1(q)= 0, 
\\& \th_2(q) =\sum_{n \equiv 1 (2)} q^\frac{n^2}{4}= 2q^\frac{1}{4} \frac{(q^4;q^4)_\infty^2}{(q^2;q^2)_\infty} =2\frac{\y^2(q^4)}{\y(q^2)},  \notag
\\& \th_3(q) =\sum_{n \in\Z} q^{n^2}=\frac{(q^2;q^2)_\infty^5}{(q;q)_\infty^2(q^4;q^4)_\infty^2}=\frac{\y^5(q^2)}{\y^2(q)\y^2(q^4)}, \notag
\\& \th_4(q) =\sum_{n \in\Z} (-1)^n q^{n^2} = \frac{(q;q)_\infty^2}{(q^2;q^2)_\infty}=\frac{\y^2(q)}{\y(q^2)}. \notag
\end{align}
From \eqref{triple} and \eqref{const2}, one obtains Jacobi's formula
\begin{equation}\label{e2.6}
\frac{\pd}{\pd z} \th_1(z,q)|_{z=0}
=\th_2(q)\th_3(q)\th_4(q)=  2\pi \y^3(q^2).
\end{equation}

We have the Lambert series expansion \cite{dirichlet} 
\begin{equation}\label{lambert2}
\th_3^2(q) =1 + 4 \sum_{n=1}^\infty \frac{\chi_4(n)q^n}{1-q^n}=1+4 \sum_{n=1}^\infty \frac{q^n-q^{3n}}{1-q^{4n}}, 
\end{equation}
where $\chi_4$ is the unique non-trivial Dirichlet character modulo $4$, extended by zero to $\Z/4\Z$.

\begin{ppn}\label{p2.1}
We have duplication formulas: 
\begin{align}
2\th_3(q^2)\th_3(2z,q^2) &= \th_3^2(z,q)+\th_4^2(z,q), \label{e2.4} 
\\2\th_2(q^2)\th_2(2z,q^2) &= \th_3^2(z,q)- \th_4^2(z,q),  \notag
\\ \th_4(q^2) \th_4(2z,q^2) &= \th_3(z,q)\th_4(z,q), \notag
\\2\th_3(q^2)\th_2(2z,q^2) &= \th_2^2(z,q)-\th_1^2(z,q), \notag
\\ 2\th_2(q^2)\th_3(2z,q^2) &= \th_2^2(z,q)+\th_1^2(z,q), \notag
\\\th_4(q^2)\th_1(2z,q^2) & = \th_1(z,q)\th_2(z,q), \notag
\end{align}
from which follows 
\begin{equation}\label{e2.7}
\th_1^4(z,q) + \th_3^4(z,q) = \th_2^4(z,q) + \th_4^4(z,q). 
\end{equation}
\end{ppn}

\begin{proof}For $i=0, 1$, we have 
\begin{align*}
\th_3^2(z,q)+(-1)^i \th_4^2(z,q) =2 \sum_{m+n\equiv i (2)}q^{m^2+n^2} u^{m+n} 
= 2 \sum_{k\equiv i (2)}q^\frac{k^2}{2} \sum_{l\equiv i (2)}q^\frac{l^2}{2} u^l, 
\end{align*}
where we set $(k,l)=(m-n,m+n)$. Similarly, we have 
\begin{align*}
&\th_3(z,q)\th_4(z,q) = \sum_{m,n} (-1)^n q^{m^2+n^2} u^{m+n}
\\&= \sum_{k\equiv 0(2)} (-1)^\frac{k}{2} q^\frac{k^2}{2} \sum_{l\equiv 0(2)} (-1)^\frac{l}{2} q^\frac{l^2}{2}u^l
+\sum_{k\equiv 1(2)} (-1)^\frac{k-1}{2} q^\frac{k^2}{2}  \sum_{l\equiv 1(2)} (-1)^\frac{l-1}{2} q^\frac{l^2}{2}u^l,
\end{align*}
and the second term of the right hand side vanishes by $\th_1(q^2)=0$. 
These prove the first three formulas, to which the other three are equivalent by \eqref{e2.2}. 
(The third and the sixth also follow from \eqref{triple}.)
\end{proof}

Letting $z=0$ in the formulas as above, we obtain
\begin{align}
2\th_3^2(q^2) &=\th_3^2(q)+\th_4^2(q),\label{e2.8}
\\2\th_2^2(q^2)& =\th_3^2(q)-\th_4^2(q), \notag
\\\th_4^2(q^2)&=\th_3(q)\th_4(q), \notag
\\ 2\th_2(q^2)\th_3(q^2) &=\th_2^2(q), \notag
\end{align}
and the Fermat quartic equation 
\begin{equation*}
\th_2^4(q)+\th_4^4(q)=\th_3^4(q).
\end{equation*}
By \eqref{e2.8}, the sequence $(\th_3^2(q^{2^n}), \th_4^2(q^{2^n}))$
is a quadratic AGM sequence (see Introduction).  

From \eqref{e2.4} and \eqref{e2.8}, one derives quadratic relations 
\begin{align}
&\th_2^2(q) \th_1^2(z,q) = \th_3^2(q) \th_4^2(z,q)- \th_4^2(q) \th_3^2(z,q) , \label{e2.10}
\\& \th_2^2(q) \th_2^2(z,q) = \th_3^2(q)\th_3^2(z,q)- \th_4^2(q) \th_4^2(z,q), \notag 
\\&  \th_2^2(q)\th_3^2(z,q) = \th_3^2(q) \th_2^2(z,q)+ \th_4^2(q) \th_1^2(z,q), \notag 
\\&  \th_2^2(q)\th_4^2(z,q) = \th_3^2(q)\th_1^2(z,q)+ \th_4^2(q) \th_2^2(z,q). \notag 
\end{align}

\begin{rmk}
The formulas \eqref{e2.7} and \eqref{e2.10} are special cases of Riemann's relation
\begin{equation*}
\prod_{i=1}^4 \th_1(x_i,q)+\prod_{i=1}^4 \th_2(x_i,q)+\prod_{i=1}^4 \th_3(x_i,q)+\prod_{i=1}^4 \th_4(x_i,q)
=2 \prod_{i=1}^4 \th_3(y_i,q),  
\end{equation*}
where 
\[\pmat{y_1\\y_2\\y_3\\y_4}=\frac{1}{2}\pmat{1&1&1&1\\1&1&-1&-1\\1&-1&1&-1\\1&-1&-1&1}\pmat{x_1\\x_2\\x_3\\x_4}\]
 (cf. \cite[Chapter I, \S5]{tata}). 
 Letting $(x_1,x_2,x_3,x_4)=(z,z,z,-z))$, we obtain \eqref{e2.7}. 
Letting $(x_1,x_2,x_3,x_4)=(z,z,0,0)$, we obtain the second formula of \eqref{e2.10},  and the others are its translations. 
\end{rmk}

The following is a restatement of differentiation formulas of Jacobi's elliptic functions $\cn$, $\dn$ and $\sn$ (see \cite[Chapter XXII]{ww}).  
\begin{ppn}We have
\begin{align}
\frac{\pd}{\pd z} \frac{\th_3(z,q)}{\th_4(z,q)} &= -\pi \th_2^2(q) \frac{\th_1(z,q)\th_2(z,q)}{\th_4^2(z,q)}, \label{dn'}
\\\frac{\pd}{\pd z} \frac{\th_2(z,q)}{\th_1(z,q)} &= -\pi \th_2^2(q) \frac{\th_3(z,q)\th_4(z,q)}{\th_1^2(z,q)}. \notag
\end{align}
\end{ppn}
\begin{proof}
As in the proof of Proposition \ref{p2.1}, 
\begin{align*}
&\frac{\pd}{\pd z} \frac{\th_3(z)}{\th_4(z)}= \frac{2\pi i}{\th_4^2(z)}\left(\Bigl(u\frac{\pd}{\pd u}\th_3(z)\Bigr) \th_4(z) - \th_3(z) \Bigl(u\frac{\pd}{\pd u}\th_4(z)\Bigr)\right)
\\& = \frac{2\pi i}{\th_4^2(z)} \left(\sum_{k\equiv l\equiv 0(2)}(-1)^\frac{k}{2} k q^\frac{k^2}{2}  (-1)^\frac{l}{2}  q^\frac{l^2}{2} u^l 
+
\sum_{k\equiv l\equiv 1(2)}(-1)^\frac{k-1}{2} kq^\frac{k^2}{2} (-1)^\frac{l-1}{2}  q^\frac{l^2}{2} u^l \right). 
\end{align*}
The first sum vanishes since $\sum_{k\equiv 0(2)}(-1)^\frac{k}{2} kq^\frac{k^2}{2}=0$ by the skew symmetry for $k \mapsto -k$. 
For the second sum, we have by \eqref{e2.6} and \eqref{e2.8} 
\[\sum_{k\equiv 1(2)}(-1)^\frac{k-1}{2} k q^\frac{k^2}{2} 
=\th_2(q^2)\th_3(q^2)\th_4(q^2)=\th_2^2(q)\th_4(q^2)/2,\]
and by \eqref{e2.4} 
\[\sum_{l\equiv 1(2)}(-1)^\frac{l-1}{2} q^\frac{l^2}{2}u^l 
=i \th_1(2z,q^2)=i{\th_1(z,q)\th_2(z,q)}/{\th_4(q^2)},\]
so we are done. 
\end{proof}

\subsection{Cubic theta functions}\label{s2.3}

First, we recall the cubic theta constants defined by Borwein-Borwein \cite{bb}
\begin{align*}
a(q) &= \sum_{m,n\in\Z} q^{m^2+mn+n^2},
\\b(q) &= \sum_{m,n\in\Z} \o^{m-n} q^{m^2+mn+n^2}, 
\\c(q) &= \sum_{m,n\in\Z} q^{(m+\frac{1}{2})^2+(m+\frac{1}{2})(n+\frac{1}{2})+(n+\frac{1}{2})^2}.
\end{align*}
These satisfy the Fermat cubic  equation \cite[(2.3)]{bb} (see also \cite[Theorem 2.3]{bbg})  
\begin{equation}\label{abc}
b^3(q)+c^3(q) = a^3(q).
\end{equation}
The series $a(q), b(q), c(q)$ are cubic analogues of $\th_3^2(q), \th_4^2(q), \th_2^2(q)$, respectively: 
note that
\[a(q) = \sum_{x \in \Z[\o]} q^{|x|^2}, \quad \th_3^2(q) = \sum_{x \in \Z[i]} q^{|x|^2}.\]
The former three are modular forms of weight one and of level three, where the latter three are of the same weight and of level two.  

One may also define $b(q)$ and $c(q)$ from $a(q)$ by 
\[b(q)= \frac{3a(q^3)-a(q)}{2}, \quad c(q)= \frac{a(q^\frac{1}{3})-a(q)}{2}\]
\cite[Lemma 2.1]{bbg}. 
Since 
\[a(q^3)=\frac{a(q)+2b(q)}{3}, \quad c(q^3)=\frac{a(q)-b(q)}{3}, \] 
and by \eqref{abc}, the sequence $(a(q^{3^n}), b(q^{3^n}))$ 
is a cubic AGM sequence (see Introduction). 

Similarly to \eqref{const2} and \eqref{lambert2}, we have product representations \cite[Proposition 2.2]{bbg} 
\begin{equation}\label{const3}
b(q)=\frac{(q;q)_\infty^3}{(q^3;q^3)_\infty}=\frac{\y^3(q)}{\y(q^3)}, \quad 
c(q)=3q^\frac{1}{3} \frac{(q^3;q^3)_\infty^3}{(q;q)_\infty}=3 \frac{\y^3(q^3)}{\y(q)}, 
\end{equation}
and the Lambert series expression \cite{dirichlet} (see also \cite[Theorem 2.3 (b)]{bb}, \cite[(2.21)]{bbg}) 
\begin{equation}\label{lambert3}
a(q)=1+6 \sum_{n=1}^\infty \frac{\chi_3(n)q^n}{1-q^n} = 1 + 6 \sum_{n=1}^\infty \frac{q^n-q^{2n}}{1-q^{3n}},
\end{equation}
where $\chi_3$ is the unique non-trivial Dirichlet character modulo $3$, extended by zero to $\Z/3\Z$. 

The {\it level three theta functions} are defined by 
\[\wt\Th_k(z,q)= \sum_{n\in\Z, n \equiv k (3)} q^\frac{n^2}{3}u^n \quad (k\in\Z/3\Z).\]
One easily shows  the quasi-periodicity 
\begin{align}
\wt\Th_k(z+1,q)&=\wt\Th_k(z,q), &\wt\Th_k(z+\t,q)&=(qu)^{-3} \wt\Th_k(z,q), \label{e2.14}
\\ \wt\Th_k(z+1/3,q)&=\o^k \wt\Th_k(z,q), & \wt\Th_k(z+\t/3,q)&=(q^{\frac{1}{3}}u)^{-1} \wt\Th_{k+1}(z,q). \notag
\end{align}
In terms of Jacobi theta functions, we have 
\begin{align}\label{e2.18}
\wt\Th_k(z,q) & =q^\frac{k^2}{3}u^k \th_3(3z+k\t,q^3),
\end{align}
where $k$ on the right hand side is any representative in $\Z$.   
We normalize $\wt\Th_k$ as follows. 

\begin{dfn}\label{d2.5}
For $k\in\Z/3\Z$, define
\[\Th_k(z,q)
=\o^k \wt\Th_k\left(z+\frac{1+3\t}{6},q\right) 
=  \sum_{n\in\Z, n \equiv k (3)} (-1)^nq^{\frac{n^2}{3}+n}u^n.\]
\end{dfn}

Then the formulas \eqref{e2.14} and \eqref{e2.18} translate to:
\begin{align}
\Th_k(z+1,q) &= \Th_k(z,q), &  \Th_k(z+\t,q) &= -(q^2u)^{-3} \Th_k(z,q),  \label{e2.30}\\
\Th_k(z+1/3,q) &= \o^k \Th_k(z,q), & \Th_k(z+\t/3,q)&= - (q^\frac{4}{3}u)^{-1} \Th_{k+1}(z,q),  \notag
\end{align}
and 
\begin{equation}\label{e2.31}
\Th_k(z,q)=i(-1)^k q^{\frac{k^2}{3}-\frac{3}{4}} u^{k-\frac{3}{2}}  \th_1(3z+k\t,q^3), 
\end{equation}
where $k$ on the right hand side is as before. 
In particular, $\Th_k(z,q)$ is an entire function with simple zero at 
\[z\equiv \frac{l-k\t}{3} \pmod{\langle 1, \t \rangle} \quad (l=0,1,2).\]

\begin{ppn}\label{p2.7}\ 
\begin{enumerate}
\item For $k=0,1,2$, we have 
\begin{equation}\label{e2.28}
\Th_k(z,q)=(-1)^{k+1} q^{\frac{k^2}{3}-k} u^{k-3} (q^6;q^6)_\infty (u^3q^{2k};q^6)_\infty (u^{-3}q^{2(3-k)};q^6)_\infty.  
\end{equation}
\item 
We have 
\begin{align}\label{e2.15}
\Th_0(0,q)=0, \quad \Th_1(0,q)= q^{-\frac{3}{4}} \y(q^2), \quad  \Th_2(0,q) =-q^{-\frac{3}{4}} \y(q^2). 
\end{align}
\item For $k=1, 2$, we have
\begin{equation}\label{e2.19}
D\log\Th_k(z,q)|_{z=0} = \frac{3+(-1)^k a(q^2)}{2}, 
\end{equation}
where $D \coloneqq (2\pi i)^{-1}\frac{\pd}{\pd z}=u\frac{\pd}{du}$. 
\end{enumerate}
\end{ppn}
\begin{proof}
The statement (i) is immediate from \eqref{e2.31} and \eqref{triple}, 
from which (ii) follows using 
$(q^6;q^6)_\infty(\o q^6;q^6)_\infty(\o^2 q^6;q^6)_\infty= (q^2;q^2)_\infty$.
Taking the logarithmic derivative of \eqref{e2.28}, we have 
\begin{align*}
&D\log\Th_k(z,q) = k- 3 - \sum_{n=0}^\infty \frac{3u^3q^{2k+6n}}{1-u^3q^{2k+6n}} 
+\sum_{n=0}^\infty \frac{3u^{-3}q^{2(3-k)+6n}}{1-u^{-3}q^{2(3-k)+6n}}
\\& = k- 3 - 3 \left(\sum_{n \ge 0, n \equiv k (3)} \frac{u^3q^{2n}}{1-u^3q^{2n}} - \sum_{n\ge 1, n \equiv 3-k (3)} \frac{u^{-3}q^{2n}}{1-u^{-3}q^{2n}}\right). 
\end{align*}
Letting $u=1$, we obtain (iii) by  \eqref{lambert3}. 
\end{proof}

We will reduce some relations among $\Th_k(z,q)$ to properties of $\th_1(z,q)$. 
First, by \eqref{triple},  \eqref{const3}, \eqref{lambert3} and a similar argument as in the proof of Proposition \ref{p2.7} for the derivatives, 
we have the following formulas: 
\begin{align}
&\th_1(1/3,q)=\th_1(2/3,q)=\sqrt{3} \y(q^6), \label{e2.20}
\\& \th_1(\t/3,q)=iq^{-\frac{1}{9}}\y(q^\frac{2}{3}), \quad
 \th_1(2\t/3)=iq^{-\frac{4}{9}}\y(q^\frac{2}{3}), \label{e2.21}
\\& \th_1(z,q)\th_1(z+1/3,q)\th_1(z+2/3,q)=b(q^2) \th_1(3z,q^3), \label{e2.23}
\\& \th_1(z,q)\th_1(z+\t/3,q)\th_1(z+2\t/3,q)=-\frac{1}{3}q^{-\frac{5}{9}}u^{-1}c(q^\frac{2}{3}) \th_1(z,q^\frac{1}{3}),  \label{e2.24}
\\&\frac{\pd}{\pd z} \th_1(z,q)\bigr|_{z=\frac{1}{3}} = -\frac{\pd}{\pd z} \th_1(z,q)\bigr|_{z=\frac{2}{3}}
=\pi\y(q^6) a(q^2), \label{e2.22.0}
\\&\frac{\pd}{\pd z} \th_1(z,q)\bigr|_{z=\frac{\t}{3}} = \frac{\pi}{3}q^{-\frac{1}{9}}\y(q^\frac{2}{3}) \bigl(2+a(q^\frac{2}{3})\bigr), \label{e2.22}
\\&
\frac{\pd}{\pd z} \th_1(z,q)\bigr|_{z=\frac{2\t}{3}} = \frac{\pi}{3} q^{-\frac{4}{9}}\y(q^\frac{2}{3})\bigl(4-a(q^\frac{2}{3})\bigr). \notag
\end{align}

Concerning cubic properties of $\th_1(z,q)$, the following result is fundamental.
 
\begin{thm}[Liu {\cite[Theorem 1]{liu}}]\label{liu} Let $f(z)$ be an entire function satisfying 
\[f(z+1)=-f(z), \quad f(z+\t)=-(qu)^{-3} f(z), \]
and put 
\[F(x,y,z)=(f(x)-f(-x)) \th_1(y,q)\th_1(z,q)\th_1(y+z,q)\th_1(y-z,q).\]
Then we have 
\[F(x,y,z)+F(y,z,x)+F(z,x,y)=0.\]\qed
\end{thm}

Note that the notations $\th_1(z|\t)$ and $q$ in \cite{liu} are respectively $\th_1(z/\pi,q)$ and $q^2$ of the present paper.
This theorem implies the two corollaries below. 
First, apply the theorem to $f(z)=\th_1^3(z+1/3)$ (resp. $f(z)=u\th_1^3(z+\t/3,q)$). 
Applying $\frac{\pd}{\pd x}$ to the 
resulting equation and letting $(x,y)=(0,1/3)$ and $(0,\t/3)$ respectively, 
we obtain the following. 

\begin{cor}[special cases of {\cite[Theorem 3]{liu}}]We have
\begin{align}
 \th_1^3(z,q) - \th_1^3(z+1/3,q) +\th_1^3(z+2/3,q) & = -3 a(q^2) \th_1(3z,q^3), \label{e2.25.0}
\\ \th_1^3(z,q^3)-qu \th_1^3(z+\t,q^3) +q^4u^2 \th_1^3(z+2\t,q^3) & =a(q^2)\th_1(z,q). \label{e2.25}
\end{align}
\qed
\end{cor}

Secondly, applying the theorem to $f(z)=\th_1^3(z,q)$ and letting $(x,y)=(1/3,1/3)$ and $(1/3,\t/3)$ respectively, we obtain the following. 
\begin{cor}[special cases of {\cite[Theorem 4]{liu}}]We have
\begin{align}
c(q^6) \th_1(z,q) &= \th_1^3(z,q^3) + \frac{b(q^2)}{b(q^6)} \th_1(z,q^3)\th_1(z+1/3,q^3)\th_1(z+2/3,q^3), \label{e2.26}
\\ 
-b(q^\frac{2}{3}) \th_1(3z,q^3) &= \th_1^3(z,q) +3q^\frac{5}{9}u \frac{c(q^2)}{c(q^\frac{2}{3})} \th_1(z,q)\th_1(z+\t/3,q)\th_1(z+2\t/3,q). \label{e2.27}
\end{align}
\qed
\end{cor}

\section{Jacobi and Legendre Elliptic curves}\label{s3}

In this section, we give uniformizations of Jacobi and Legendre elliptic curves using quadratic theta functions, and 
give isogenies corresponding to the quadratic AGM step. 

\subsection{Jacobi elliptic curves}\label{s3.1}

For $\k \ne 0, \pm 1$, let $J(\k)$ be the {\it Jacobi elliptic curve} defined by the affine equation
\[y^2=(1-x^2)(1-\k^2 x^2).\]
More precisely, we resolve the singularity at infinity to obtain the equation 
\[t^2=(s^2t^2-1)(s^2t^2-\k^2), \quad s=x/y, \quad t=y/x^2.\]
The points at infinity are $(s,t)=(0,\pm \k)$. We take $(s,t)=(0,\k)$ as the origin. 
Put  
\[\o_{J(\k)}=\frac{dx}{y},\]
a holomorphic $1$-form on $J(\k)$. 

\begin{thm}\label{t3.1} 
Let $\k(q)=\th_3^2(q)/\th_4^2(q)$. 
Then the map
\[\Phi_\t \colon \C/\langle 1,2\t\rangle \to J(\k(q));\]
\[ 
z \mapsto \left(i\frac{\th_4(q)}{\th_3(q)} \cdot \frac{\th_2(z,q)}{\th_1(z,q)}, 
-\frac{\th_2^2(q)}{\th_3(q)\th_4(q)} \cdot \frac{\th_3(z,q)\th_4(z,q)}{\th_1^2(z,q)}\right)
\]
is a well-defined isomorphism of elliptic curves. 
It sends the $2$-torsion points as follows: 
\begin{alignat*}{2}
0 & \mapsto (s,t)=(0,\k(q)), & 1/2 & \mapsto (x,y)=(0,-1), 
\\\tau & \mapsto (s,t)=(0,-\k(q)), \quad & 1/2+\tau & \mapsto (x,y)=(0,1). 
\end{alignat*}
Moreover, we have
\[\Phi_\t^*(\o_{J(\k(q))})=\pi i \th_4^2(q) dz.\]
\end{thm}

\begin{proof}
The map $\Phi_\t$ factors through $\C/\langle 1, 2\t\rangle$ by \eqref{e2.1}, 
and its image lies in $J(\k(q))$ by \eqref{e2.10}. 
Since 
\[(s,t)= \left(-i \frac{\th_4^2(q)}{\th_2^2(q)}\cdot \frac{\th_1(z,q)\th_2(z,q)}{\th_3(z,q)\th_4(z,q)},  \frac{\th_2^2(q)\th_3(q)}{\th_4^3(q)}\cdot \frac{\th_3(z,q)\th_4(z,q)}{\th_2^2(z,q)}\right),\]
the images of $0$ and $\t$ are determined. Use \eqref{e2.2} for the other values. 
It is easy to see that $\Ker(\Phi_\t)$ is trivial, hence $\Phi_\t$ is injective. Being an injective map of complex tori sending origin to origin, $\Phi_\t$ is an isomorphism. 
The last statement follows by \eqref{dn'}. 
\end{proof}

Recall from Section \ref{s2.1} that $\k(q^{2^n})$ is the ratio of a quadratic AGM sequence. 

\begin{thm}\label{t3.2}
We have a commutative diagram
\[\xymatrix{
\C/\langle 1, 2\t \rangle \ar[d]_{\varphi} \ar[r]^{\Phi_\t} & J(\k(q))  \ar[d]_{{f}}
\\
\C/\langle 1, 4\t \rangle \ar[r]^{\Phi_{2\t}} & J(\k(q^2))
}\]
where $\varphi(z)=2z$, and ${f}$ is an  isogeny of degree two defined by 
\[f(x,y)=\left(\frac{1+\k(q) x^2}{(1+\k(q))x}, -\frac{y(1-\k(q) x^2)}{2\sqrt{\k(q)} (1+\k(q))x^2}\right), \quad \sqrt{\k(q)} \coloneqq \frac{\th_3(q)}{\th_4(q)}.\]
Moreover, we have 
\[{f}^*(\o_{J(\k(q^2))}) = 2\sqrt{\k(q)}\cdot  \o_{J(\k(q))}.\]
\end{thm}

\begin{proof}
Put $(x',y')=\Phi_{2\t}(2z)$. Then, by \eqref{e2.4},  \eqref{e2.8} and \eqref{e2.10}, we have 
\begin{align*}
x'&=i \frac{\th_4(q^2)}{\th_3(q^2)} \cdot \frac{\th_2(2z,q^2)}{\th_1(2z,q^2)}
\\& = i \frac{\th_4^2(q^2)}{2\th_2(q^2)\th_3(q^2)} \cdot \frac{\th_3^2(q,z)-\th_4^2(z,q)}{\th_1(z,q)\th_2(z,q)}
\\&= i \frac{\th_3(q)\th_4(q)(\th_3^2(q)-\th_4^2(q))}{\th_2^4(q)}\cdot \frac{\th_2^2(q,z)-\th_1^2(z,q)}{\th_1(z,q)\th_2(z,q)}, 
\end{align*}
and similarly
\begin{align*}
y'&= - \frac{\th_2^2(q^2)}{\th_3(q^2)\th_4(q^2)} \cdot \frac{\th_3(2z,q^2)\th_4(2z,q^2)}{\th_1^2(2z,q^2)}
\\&= - \frac{\th_3^2(q)-\th_4^2(q)}{2(\th_3^2(q)+ \th_4^2(q))} \cdot \frac{(\th_3^2(z,q)+\th_4^2(z,q))\th_3(z,q)\th_4(z,q)}{\th_1^2(z,q)\th_2^2(z,q)}
\\&= - \frac{\th_2^2(q)}{2(\th_3^2(q)+ \th_4^2(q))} \cdot \frac{\th_1^2(z,q)+\th_2^2(z,q)}{\th_2^2(q)} \cdot \frac{\th_1^2(z,q)}{\th_3(z,q)\th_4(z,q)}. 
\end{align*}
This yields the formula of $f(x,y)$.
The last statement follows since
\begin{align*}
&\Phi_\t^* f^* (\o_{\k(q^2)}) =\varphi^*\Phi_{2\t}^*(\o_{\k(q^2)}) = \varphi^*(\pi i\th_4^2(q^2) dz)=2\pi i \th_4^2(q^2) dz 
\\&= 2 \pi i \th_3(q)\th_4(q) dz
= \Phi_\t^*\left(2\sqrt{\k(q)} \cdot \o_{\k(q)}\right),
\end{align*}
and $\Phi_\t^*$ is injective. 
\end{proof}

\subsection{Legendre elliptic curves}\label{s3.2}

For $\l \ne 0, 1$, the {\it Legendre elliptic curve} $L(\l)$ is defined by the homogeneous equation 
\[V^2W = U(W-U)(W-\l U),\]
and we take $[0:1:0]$ as the origin. 
Setting $u=U/W$, $v=V/W$, its affine equation is
\[v^2=u(1-u)(1- \l u).\]
Put a holomorphic $1$-form on $L(\l)$ as
\[\o_{L(\l)} = \frac{du}{v}. \]

There exists an isogeny of degree $2$ 
\[\pi\colon J(\k) \to L(\k^2); \quad [U:V:W]=[x^2:xy:1]=[s:1:s^3t^2], \] 
whose kernel is given by $(s,t)=(0, \pm \k)$.  Note that 
\[\pi^* \o_{L(\k^2)} = \frac{dx^2}{xy}=2 \o_{J(\k)}.\] 
The following is a direct consequence of Theorems \ref{t3.1} and \ref{t3.2}. 

\begin{thm}\label{t3.3}\ 
\begin{enumerate}
\item
The map 
\[\Psi_\t \colon \C/\langle 1,\t\rangle \to L(\k^2(q));\]
\[z \mapsto \left[\th_4^2(q) \th_3^2(z,q)\th_4(z,q): 
i\th_2^2(q) \th_1(z,q)\th_2(z,q)\th_3(z,q): \th_3^2(q)\th_4^3(z,q)\right]\]
is a well-defined isomorphism of elliptic curves, 
which makes the diagram 
\[\xymatrix{
\C/\langle 1, 2\t \rangle \ar[d]_\varpi \ar[r]^{\Phi_\t} & J(\k(q))  \ar[d]_\pi
\\
\C/\langle 1, \t \rangle \ar[r]^{\Psi_{\t}} & L(\k^2(q))
}\]
commute, where $\varpi(z)=z$. 
It sends the $2$-torsion points as follows: 
\begin{alignat*}{2}
0 & \mapsto [0:1:0], & 1/2 & \mapsto [0:0:1], 
\\\tau/2 & \mapsto [1:0:1], \quad  & (1+\tau)/2 & \mapsto [1:0:\k^2(q)]. 
\end{alignat*}
Moreover, we have 
\[\Psi_\t^* \o_{L(\k^2(q))} = 2\pi i \th_4^2(q) dz\]
\item
We have a commutative diagram
\[\xymatrix{
\C/\langle 1, \t \rangle \ar[d]_{\psi} \ar[r]^{\Psi_\t} & L(\k^2(q))  \ar[d]_{g}
\\
\C/\langle 1, 2\t \rangle \ar[r]^{\Psi_{2\t}} & L(\k^2(q^2))
}\]
where $\psi(z)=2z$,  and $g$ is an isogeny defined by 
\[g(u,v)
=\left(\frac{(1+\k(q) u)^2}{(1+\k(q))^2u}, \frac{(1-\k^2(q)u^2)v}{2\sqrt{\k(q)}(1+\k(q))^2u^2}
\right), 
\]
where $\sqrt{\k(q)}$ is as in Theorem \ref{t3.2}.  
Moreover, we have 
\[g^*(\o_{L(\k^2(q^2))})=2\sqrt{\k(q)} \cdot  \o_{L(\k(q))}.\]
\end{enumerate}
\qed
\end{thm}

\section{Hesse elliptic curves}\label{s4}

\subsection{Hesse Elliptic curves}\label{s4.1}

First, we collect algebraic properties of Hesse elliptic curves. 
Let $K$ be a field of characteristic different from $3$, and $\ol K$ be its algebraic closure. 
For $\m \in K-\bmm_3$,  the {\it Hesse elliptic curve} $H(\m)$ is defined by the homogeneous equation
\begin{equation}\label{e.hesse}
X^3+Y^3+Z^3-3\m XYZ =0.
\end{equation}
It is non-singular by the assumption $\m \not\in \bmm_3$. 
We take $O\coloneqq[0:-1:1]$ as the origin. Let 
\[x^3+y^3+1-3\m xy=0, \]
where $(x,y)=(X/Z,Y/Z)$,  be the affine equation of $H(\m)$. Then
\[\o_\m \coloneqq  \frac{dy}{x^2 -\m y} =-\frac{dx}{y^2-\m x} \]
is a non-trivial holomorphic $1$-form on $H(\m)$. 

The group law of the elliptic curve $H(\m)$ is extremely simple (cf. \cite[(8), (9), (10) and (13)]{jq}): 
\begin{itemize}
\item
The inversion
\begin{equation}\label{e3.1}
-[X:Y:Z] =[X:Z:Y],
\end{equation}
\item The addition of distinct points 
\begin{align}\label{e3.2}
&[X:Y:Z] + [X':Y':Z'] 
\\&= [X^2Y'Z'-X'^2YZ: Z^2X'Y'-Z'^2XY:Y^2X'Z'-Y'^2XZ], \notag
\end{align}
\item The doubling 
\begin{align}\label{e3.3}
2[X:Y:Z]&=[Y:Z:X]+[Z:X:Y]
\\&=[X(Y^3-Z^3):Z(X^3-Y^3):Y(Z^3-X^3)]. \notag
\end{align}
\end{itemize}

The $3$-torsion points of $H(\m)$ are given by $XYZ=0$, i.e. 
\[[0:-1:\z], \ [-1:\z:0], \ [\z:0:-1] \quad (\z\in\bmm_3). \]
It follows from \eqref{e3.3} that the $2$-torsion points of $H(\m)$ are $[0:-1:1]$ and 
\[\left[-\z \sqrt[3]{1-\sqrt{1-\m^3}}-\z^{-1} \sqrt[3]{1+\sqrt{1-\m^3}}:1:1\right] \quad (\z\in\bmm_3),\]
where the product of two cubic roots is $\m$ (use Cardano's formula). When $\mathrm{char}(K)=2$, these reduce to a single point, i.e. $H(\m)$ is supersingular, if and only if $\m=0$. 

\begin{rmk}\label{r4.1}
By our choice of $O$, the $3$-torsion points of $H(\m)$ are the inflection points. 
There are twelve lines connecting three inflection points: 
\[X=0, \ Y=0, \ Z=0, \ \z X+ \z' Y+Z=0 \ (\z, \z' \in \bmm_3).\]
Let $E$ be an elliptic curve over $K$ whose $3$-torsion points are all defined over $K$.  
Note that $\bmm_3 \subset K$ by the Weil pairing (cf. \cite[III, \S8]{silverman}). 
If $E[3]$ is generated by $P$ and $Q$, then there exists a unique $\m\in K-\bmm_3$ and a unique $K$-isomorphism $E \simeq H(\m)$ sending $P$ and $Q$ to $[-1:1:0]$ and $[0:-1:\o]$, respectively, where $\o\in\bmm_3$ is a fixed primitive root. 
Under the isomorphism, the line $X+Y+Z=0$ (resp. $X=0$) on $H(\m)$ corresponds to the line on $E$ connecting 
$\{O,P,2P\}$ (resp. $\{O,Q,2Q\}$). 
 
\end{rmk}

Suppose now that $K$ contains $\bmm_3$ and fix a primitive cubic root of unity $\o$. 
Automorphisms $R$ and $S$ of $H(\m)$ (as a $K$-variety) of order three are defined by
\begin{align*}
R([X:Y:Z])&=[Y:Z:X], \\ S([X:Y:Z])&=[X:\o Y: \o^2 Z].
\end{align*}
In fact, these are the translations by the $3$-torsion points $[-1:1:0]$ and $[0:-1:\o]$, respectively. 
In particular, $R\circ S=S \circ R$ and 
$R^*\o_\m=S^*\o_\m=\o_\m$ hold. 

Define linear fractional transformations respectively of order $2$ and $3$ by 
\[\a(\m)=\frac{\m+2}{\m-1}, \quad \b(\m)=\o\m.\]
These send the four points $\bmm_3 \cup \{\infty\}$ to themselves, and the group 
$G\coloneqq \langle \a, \b \rangle$ 
is isomorphic to the alternating group of degree $4$. 
We will use the presentation
\[G=\left\{\b^i, \b^i\a\b^j \mid i, j \in \Z/3\Z\right\}.\]
By the next proposition, the $G$-action does not change the isomorphism class of $H(\m)$. 

\begin{ppn}\label{p3.6}For any $\m \in K-\bmm_3$, we have isomorphisms of elliptic curves  
\[A_\m \colon H(\m) \to H(\a(\m)); \quad 
[X:Y:Z]\mapsto [X:Y:Z]\bmat{1&1&1\\1& \o & \o^2\\1&\o^2&\o} 
\]
and
\[B_\m\colon H(\m) \to H(\b(\m)); \quad [X:Y:Z] \mapsto [\o^{-1} X:Y:Z].\]
These satisfy
\begin{alignat}{2}
A_\m \circ S &= R \circ A_\m, \quad & A_\m \circ R &= S^{-1} \circ A_\m,\label{e4.4}
\\B_\m \circ S &= S \circ B_\m, \quad & B_\m \circ R &= R \circ B_\m.\notag
\end{alignat}
Moreover, we have 
\begin{equation}\label{e4.6}
A_\m^*(\o_{\a(\m)}) = \frac{i}{\sqrt{3}} (\m-1) \o_\m, \quad 
B_\m^* (\o_{\b(\m)}) = \o^{-1} \o_\m.
\end{equation}
\end{ppn}

\begin{proof}
It is straight-forward to verify that the maps $A_\m$ and $B_\m$ are well-defined and send the origin to the origin. The composition $A_{\a(\m)}\circ A_\m$ is the inversion by \eqref{e3.1}
since 
\[\bmat{1&1&1\\1& \o & \o^2\\1&\o^2&\o}^2=3\bmat{1&0&0\\0& 0 & 1\\0&1&0},\]
and $B_{\b^2(\m)}\circ B_{\b(\m)} \circ B_\m$ is the identity. 
Hence $A_\m$ and $B_\m$ are isomorphisms. 
The proofs of \eqref{e4.4} and \eqref{e4.6} are straight-forward. 
\end{proof}

It is sometimes convenient to convert $H(\m)$ to a Weierstrass canonical form (see \cite[Chapter III, \S1]{silverman}).  
There is an isomorphism from $H(\m)$ to the Weierstrass elliptic curve
\begin{equation}\label{weier1}
E(\m) : y^2z+(\m+2) xyz +\frac{\m^2+\m+1}{3} yz^2 = x^3
\end{equation}
with origin $[0:1:0]$, given by 
\[[x:y:z]=\left[-(X+Y+Z):X+ Y+\m Z:  \frac{3}{\m^2+\m+1}(\m X+Y+ Z)
\right]
\]
(see for example \cite[Remark 1]{fj}). The inverse is: 
\[[X:Y:Z]=\left[x+\frac{\m^2+\m+1}{3} z:-\left((\m+1)x+y+\frac{\m^2+\m+1}{3} z\right): x+y\right]. \]
One computes the discriminant
\[\D(E(\m))=\left(\frac{\m^3-1}{3}\right)^3, \]
and the $j$-invariant
\begin{equation}\label{e4.j}
j(H(\m))=j(E(\m))=27\m^3\left(\frac{\m^3+8}{\m^3-1}\right)^3. 
\end{equation}

\begin{ppn}\label{p4.3}
For $\m, \n \in \ol K-\bmm_3$,  we have $H(\m) \simeq H(\n)$  if and only if $\n=g(\m)$ for some $g\in G$. 
\end{ppn}

\begin{proof}
Recall that $H(\m) \simeq H(\n)$  if and only if $j(E(\m))=j(E(\n))$ (cf. \cite[III, Proposition 1.4]{silverman}). 
By Proposition \ref{p3.6}, it suffices to prove the ``only if'' part. 
For any $j\in \ol K$, the equation $j(H(\m))=j$ has at most $12$ solutions. 
If $j=0$ (resp. $j=1728$), then the solutions are given by $\m^3=0, -8$ (resp. by $\m^3=(1\pm \sqrt{3})^3$), and one verifies that these solutions are in a same $G$-orbit. 
For the remaining $\m$, one sees that the twelve values $\n=g(\m)$ ($g \in G$) are all distinct, hence constitute all the solutions of $j(H(\n))=j(H(\m))$. 
\end{proof}

\subsection{Uniformizations and isogenies}

Now, we return to the case  $K=\C$. 
We will construct two complex uniformizations of the $H(\m)$, written $\Phi_\t$ and $\Psi_\t$, 
which are related by the isomorphism $A_\m$ as above. 

Theorems \ref{okada1} and \ref{okada2} below are established in \cite{okada}. Under the author's permission, we include their proofs for  completeness. 

\begin{thm}[Okada]\label{okada1} 
Set $\r(q)=a(q^2)/c(q^2)$. The map
\[\Phi_\t \colon \C/\langle 1,\t\rangle  \to H(\r(q)); \quad
 z \mapsto [\Th_0(z,q): \Th_1(z,q):\Th_2(z,q)] 
 \]
 is a well-defined isomorphism of elliptic curves. 
 It satisfies
 \begin{equation}\label{e4.7}
 \Phi_\t(z+1/3)=S (\Phi_\t(z)), \quad \Phi_\t(z+\t/3)=R (\Phi_\t(z)). 
 \end{equation}
Moreover, we have
\begin{equation}\label{e4.8}
\Phi_\t^*(\o_{\r(q)})=c(q^2) \cdot 2\pi i \, dz.
\end{equation}
\end{thm} 

\begin{proof}
The map $\Phi_\t$ factors through $\C/\langle 1,\t\rangle$ by  \eqref{e2.30} 
and its image lies in $H(\r(q))$ by \eqref{e2.31}, \eqref{e2.24} and \eqref{e2.25}. 
We have $\Phi_\t(0)=O$ by \eqref{e2.15}, hence $\Phi_\t$ is a homomorphism. 
Since  $\Th_0(z)=0$ if and only if $z \in \langle 1,\t\rangle$ by \eqref{e2.28}, 
$\Phi_\t$ is injective, hence is an isomorphism.  
The relations \eqref{e4.7} follow from \eqref{e2.14}. 
Finally we have, just writing $\Th_k=\Th_k(z,q)$,  
\begin{align*}
&\Phi_\t^*(\o_{\r(q)})
=\Phi_\t^*\left(\frac{dy}{x^2-\r(q) y}\right) 
= \frac{(\Th_1/\Th_2) d\log(\Th_1/\Th_2)}{(\Th_0/\Th_2)^2-\r(q)(\Th_1/\Th_2)}
\\&=\frac{2\pi i \cdot (\Th_1/\Th_2)}{(\Th_0/\Th_2)^2-\r(q)(\Th_1/\Th_2)} 
(D\log \Th_1 - D \log \Th_2)\, dz. 
\end{align*}
Since $\Phi_\t^*(\o_{\r(q)})$ is a constant multiple of $dz$, we can substitute $z=0$ in the ``coefficient'' of $dz$, and obtain
\[\Phi_\t^*(\o_{\r(q)})= \frac{2\pi i \cdot a(q^2)}{\r(q)} \, dz = c(q^2) \cdot 2\pi i \,  dz\]
by \eqref{e2.15} and \eqref{e2.19}. 
\end{proof}

\begin{rmk}
By Theorem \ref{okada1}, the formula \eqref{e3.2} gives an addition formula for $\Th_k(z,q)$, expressing $\Th_k(z+z',q)$ in terms of $\Th_k(z,q)$ and $\Th_k(z',q)$. 
In \cite[proof of Proposition 3.2]{kknt}, such addition formulas are 
derived from a classical formula of Jacobi, and are used to prove cubic relations among cubic theta functions. See \cite{bl} for more general relations among theta functions. 
\end{rmk}

\begin{thm}[Okada]\label{okada2} We have a commutative diagram
\[\xymatrix{
\C/\langle 1, \t \rangle \ar[d]_{{\varphi}} \ar[r]^{\Phi_\t} & H(\r(q))  \ar[d]_{{f}}
\\
\C/\langle 1, 3\t \rangle \ar[r]^{\Phi_{3\t}} & H(\r(q^3))
}\]
where $\varphi(z)=3z$, and ${f}$ is an isogeny defined by 
\[{f}([X:Y:Z]) = \left[X^3- 3\s(q) XYZ:  Y^3-3\s(q) XYZ: Z^3-3\s(q) XYZ\right],\]
\[\s(q)\coloneqq \frac{c(q^6)}{c(q^2)}=\frac{\y(q^2)\y^3(q^{18})}{\y^4(q^6)}. \]
Moreover, we have 
\[f^*(\o_{\r(q^3)}) = 3 \s(q) \o_{\r(q)}.\]
\end{thm}

\begin{proof}By \eqref{e2.27}, replacing $(z,q)$ with $(3z,q^3)$, we obtain 
\[b(q^2) \Th_0(3z,q^3)= \Th_0(z,q)^3 - 3 \s(q) \Th_0(z,q)\Th_1(z,q)\Th_2(z,q). \]
Then, replacing $z$ with $z+\t/3$ iteratively, we obtain by \eqref{e2.30} 
\begin{align*}
b(q^2) \Th_k(3z,q^3) &= \Th_k(z,q)^3 -3 \s(q) \Th_0(z,q)\Th_1(z,q)\Th_2(z,q)
\end{align*}
for any $k$, and the first assertion follows. 
The last statement follows by $\vp^*(dz)=3\,dz$ and Theorem \ref{okada1}. 
\end{proof}

As well as the isogenies $f$ and $g$ respectively in Theorems \ref{t3.2} and \ref{t3.3}, 
the isogeny $f$ in Theorem \ref{okada2} exists over any field. 
Note that 
\[3 \s(q) = \frac{3 \r(q)c(q^6)}{a(q^2)}=\frac{3 \r(q)c(q^6)}{a(q^6)+2c(q^6)}= \frac{3 \r(q)}{\r(q^3)+2},\]
and 
\begin{align*}
\r^3(q)& =\frac{a^3(q)}{a^3(q)-b^3(q)}=\left(\frac{a(q)}{a(q)-b(q)}\right)^3 \cdot\frac{(a(q)-b(q))^2}{a^2(q)+a(q)b(q)+b^2(q)}
\\& = \left(\frac{\r(q^3)+2}{3}\right)^3 \left(\frac{\r^2(q^3)+\r(q^3)+1}{3}\right)^{-1}
\end{align*}
by substituting  $\r(q^3)=(a(q)+2b(q))/(a(q)-b(q))$. 
\begin{thm}\label{alg-isogeny}
Let $K$ be a field of characteristic different from $3$. Suppose that $\r$ and $\r' \in K-\bmm_3$ satisfy
\[\r^3=\left(\frac{\r'+2}{3}\right)^3 \left(\frac{\r'^2+\r'+1}{3}\right)^{-1},\]
and put 
\[\s=\begin{cases} \frac{\r}{\r'+2} & (\r\ne 0), \\ \frac{1}{3} & (\r=0). \end{cases}\] 
Then there exists an isogeny $f \colon H(\r) \to H(\r')$ of degree $3$ defined by 
\[[X:Y:Z]\mapsto [X^3- 3\s XYZ:Y^3-3\s XYZ:Z^3-3\s XYZ ], \]
and its kernel is given by $X=0$. 
Moreover, we have
\begin{equation}\label{e4.12}
f^*(\o_{\r'}) = 3\s \o_\r.
\end{equation}
\end{thm}

\begin{proof}First, we prove the well-definedness of $f$. 
Put 
\[X'=X^3-3 \s XYZ, \quad Y'=Y^3-3 \s XYZ, \quad Z'=Z^3-3 \s XYZ, \]
and 
\[s=XYZ, \ S_1=X^3+Y^3+Z^3, \ S_2=X^3Y^3+Y^3Z^3+Z^3X^3, \ S_3=X^3Y^3Z^3.\]
Then, 
\begin{align*}
 X'^3+Y'^3+Z'^3
&=S_1^3-3S_1S_2+3S_3- 9\s s(S_1^2-2S_2) +27\s^2 s^2 S_1-81 \s^3 S_3
\\&=3\left(1+9\r^3-27(\s\r^2-\s^2\r+\s^3)\right)S_3-9(\r-2\s)sS_2
\intertext{and}
3\r' X'Y'Z' &=3\r'\left( S_3-3\s s S_2+ 9\s^2 s^2 S_1 -27 \s^3 S_3 \right)
\\&= 3\r'\left(1+27(\s^2\r-\s^3)\right) S_3 - 9 \r'\s sS_2.
\end{align*}
Here we used $S_1=3\r s$ and $S_3=s^3$. 
It is a routine to verify that the coefficients of $S_3$ and $s S_2$ in these expressions agree,  
hence $f$ is well-defined.  
The computation of $f^*(\o_{\m'})$ is lengthy but straight-forward, hence is omitted. 
The remaining statements are easy to verify. 
\end{proof}

Next, we give the other complex uniformization of $H(\m)$. 

\begin{thm}\label{t3.11} Set $\m(q)= a(q^\frac{2}{3})/b(q^\frac{2}{3})$. The map 
\[\Psi_{\t} \colon \C/ \angle{1, \t} \to H(\m(q))\]
defined by 
\[[X:Y:Z]=
\left[ \th_1(z,q^\frac{1}{3}):-\th_1(z+1/3,q^\frac{1}{3}):\th_1(z+2/3,q^\frac{1}{3})\right]
\]
is an isomorphism of elliptic curves. It satisfies 
 \begin{equation}\label{e4.9}
 \Psi_\t(z+1/3)=R (\Psi_\t(z)), \quad \Psi_\t(z+\t/3)=S^{-1} (\Psi_\t(z)). 
 \end{equation}
Moreover, we have
\begin{equation}\label{e4.10}
\Psi_{\t}^* \o_{\m(q)}= \frac{i}{\sqrt{3}} b(q^\frac{2}{3}) \cdot 2\pi i \, dz.
\end{equation}
\end{thm}

\begin{proof}
The map $\Psi_\t$ factors through $\C/\langle 1,\t\rangle$ by \eqref{e2.30}, and its image 
 lies in $H(\m(q))$  by \eqref{e2.23} and \eqref{e2.25.0}. 
The first of \eqref{e4.9} is evident and the second follows from \eqref{e2.1}. 
Just writing $\th_1(z)=\th_1(z,q^\frac{1}{3})$, we have 
\[\Psi_{\t}^*(\o_{\m(q)})= \Psi_{\t}^*\left(\frac{dy}{x^2-\m(q)y}\right)
=\frac{-\frac{\th_1(z+1/3)}{\th_1(z+2/3)} d\log\left(- \frac{\th_1(z+1/3)}{\th_1(z+2/3)}\right)}{
\frac{\th_1^2(z)}{\th_1^2(z+2/3)}+\m(q) \frac{\th_1(z+1/3)}{\th_1(z+2/3)}}. 
\]
Letting $z=0$ (see the discussion in the proof of Theorem \ref{okada1}), we have 
\begin{align*}
\Psi_{\t}^*(\o_{\m(q)})
&= -\frac{\frac{d}{dz}\log \th_1(z+1/3)|_{z=0} - \frac{d}{dz}\log \th_1(z+2/3)|_{z=0}}{\m(q)} dz
=- \frac{\frac{2\pi}{\sqrt{3}} a(q^\frac{2}{3})}{\m(q)}
\end{align*}
by \eqref{e2.20} and \eqref{e2.22.0}, hence \eqref{e4.10} is proved.  
\end{proof}

The relation between $\Phi_\t$ and $\Psi_\t$ is given as follows. 
Recall from Section \ref{s4.1} the transformation $\a$ and the isomorphism $A_\m$.  
We have 
\[\a(\r(q))=\frac{\r(q)+2}{\r(q)-1}=\frac{a(q^2)+2c(q^2)}{a(q^2)-c(q^2)}= \frac{a(q^\frac{2}{3})}{b(q^\frac{2}{3})} = \m(q).\]

\begin{ppn}\label{p4.7}
The diagram
\[
\xymatrix{\C/\angle{1,\t} \ar[r]^{\Phi_\t} \ar[dr]_{\Psi_\t} & H(\r(q)) \ar[d]^{A_{\r(q)}}
\\ & H(\m(q))}
\]
commutes, where $\Phi_\t$, $\Psi_\t$  are the isomorphisms in Theorems \ref{okada1}, \ref{t3.11}, respectively. 
\end{ppn}

\begin{proof}By definition, the first coordinate of $A_{\r(q)}\circ \Phi_\t$ is 
\[
\sum_{k=0}^2 \Th_k(z,q)= \sum_{n \in\Z} (-1)^n q^{\frac{n^2}{3}+n} u^n
=(iq^\frac{3}{4}u^\frac{3}{2})^{-1} \th_1(z,q^\frac{1}{3}). \]
Then the second and the third coordinates are, by \eqref{e2.30}, 
\begin{align*}&
\sum_{k=0}^2 \o^k \Th_k(z,q)= \sum_{k=0}^2 \Th_k(z+1/3,q)=-(iq^\frac{3}{4}u^\frac{3}{2})^{-1} \th_1(z+1/3,q^\frac{1}{3}),
\\&\sum_{k=0}^2 \o^{2k} \Th_k(z,q)= \sum_{k=0}^2 \Th_k(z+2/3,q)=(iq^\frac{3}{4}u^\frac{3}{2})^{-1}\th_1(z+2/3,q^\frac{1}{3}). 
\end{align*}
Hence the proposition follows.  
\end{proof}

\begin{thm}\label{t4.8}
We have a commutative diagram
\[\xymatrix{
\C/\angle{1,\t}\ar[d]_{\vp} \ar[r]^{\Psi_\t} & H(\m(q)) \ar[d]^g\\
\C/\angle{1,3\t} \ar[r]^{\Psi_{3\t}} & H(\m(q^3))
}\]
where $\vp(z)=3z$, and $g$ is an isogeny defined by
\[g([X:Y:Z])= \left[3 \n(q) XYZ: X^2Y+Y^2Z+Z^2X:XY^2+YZ^2+ZX^2\right], \] 
\[\n(q)\coloneqq  \frac{b(q^2)}{b(q^\frac{2}{3})}=\frac{\y^4(q^2)}{\y^3(q^\frac{2}{3})\y(q^6)}.\]
Moreover, we have
\[g^*(\o_{\m(q^3)})= 3 \n(q) \o_{\m(q)}. \]
\end{thm}

\begin{proof}By Theorem \ref{okada2} and Proposition \ref{p4.7}, it suffices to compute 
$g \coloneqq A_{\r(q^3)} \circ f \circ A_{\r(q)}^{-1}$.  
Put $[X':Y':Z']=A_{\r(q)}^{-1}([X:Y:Z])$, $[X'':Y'':Z'']=f\circ A_{\r(q)}^{-1}([X:Y:Z])$ and $[X''':Y''':Z''']=g([X:Y:Z])$. 
Write for brevity $\m=\m(q)$, $\r=\r(q)$ and $\s=\s(q)$. 
First, 
\[[X':Y':Z']=[X:Y:Z]\bmat{1&1&1\\1& \o & \o^2\\1&\o^2&\o}^{-1}
=[X:Y:Z]\bmat{1&1&1\\1& \o^2 & \o\\1&\o&\o^2}. \]
Secondly, noting 
\[X'Y'Z'=X^3+Y^3+Z^3-3XYZ=3(\m-1)XYZ,\]
we have 
\begin{align*}
[X'':Y'':Z'']=\bigl[&(X+Y+Z)^3- 9 \s(\m-1) XYZ
\\&: (X+\o^2 Y+\o Z)^3- 9 \s(\m-1) XYZ
\\&: (X+\o Y+\o^2 Z)^3- 9 \s(\m-1) XYZ
\bigr].
\end{align*}
Then one computes
\[ [X''':Y''':Z''']= \left[C XYZ: X^2Y+Y^2Z+Z^2X:XY^2+YZ^2+ZX^2\right], 
\]
where 
\begin{align*}
C=\left(\frac{\m+2}{\m-1}-3\s\right)(\m-1)=(\r-3\s)(\m-1)=\frac{b(q^2)}{c(q^2)} \cdot \frac{3 c(q^2)}{b(q^\frac{2}{3})}=3 \n(q), 
\end{align*}
hence the formula of $g$.  The last statement is immediate from \eqref{e4.10}.  
\end{proof}

Similarly as Theorem \ref{alg-isogeny}, the isogeny $g$ exists over any field. 
Note that 
\[3 \n(q) = \frac{3a(q^2)}{b(q^\frac{2}{3})\m(q^3)} = \frac{\m(q)+2}{\m(q^3)}.\]

\begin{thm}\label{alg-isogeny2}
Let $K$ be a field of characteristic different from $3$. Suppose that $\m, \m' \in K-\bmm_3$ satisfy 
\[\m'^3=\left(\frac{\m+2}{3}\right)^3 \left(\frac{\m^2+\m+1}{3}\right)^{-1},\]
and put 
\[\n= \begin{cases} \frac{\m+2}{3\m'} & (\m'\ne 0), \\ 1& (\m'=0). \end{cases}\]
Then, there exists an isogeny $g\colon H(\m) \to H(\m')$ of degree $3$ defined by 
\[[X:Y:Z] \mapsto [3 \n XYZ: X^2Y+Y^2Z+Z^2X:XY^2+YZ^2+ZX^2],\]
and its kernel is given by $X+Y+Z=0$. 
Moreover, we have 
\begin{equation}\label{e4.14}
g^*(\o_{\m'}) =3\n  \o_\m.
\end{equation}
\end{thm}

\begin{proof}
We may prove this directly, but we reduce it to Theorem \ref{alg-isogeny}. 
If we let $\m=\a(\r)$ and $\m'=\a(\r')$, then the relation between $\m$ and $\m'$ turns into the relation between $\r$ and $\r'$ as in Theorem \ref{alg-isogeny}. 
Hence we are reduced to compute $g \coloneqq A_{\r'} \circ f \circ A_\r^{-1}$, which is straight-forward and is omitted. By \eqref{e4.6} and \eqref{e4.12}, we have 
\[g^*(\o_{\m'}) = 3\s \frac{\r'-1}{\r-1} \o_\m = 3\n \o_\m.\]
Here, the last equality is trivial when $\m'=0$, and otherwise we have 
\[3 \s \frac{\r'-1}{\r-1}= \frac{3 \r}{\r-1} \cdot \frac{\r'-1}{\r'+2}
= \frac{3 \a(\m)}{\a(\m)-2} \cdot \frac{1}{\m'} =\frac{\m+2}{\m'}=3\n.\]
Hence the proof is complete. 
\end{proof}

\begin{cor}The dual isogenies $\check{f}\colon H(\r')\to H(\r)$ and $\check{g} \colon H(\m')\to H(\m)$ 
of $f$ and $g$ in Theorems \ref{alg-isogeny} and \ref{alg-isogeny2} are given respectively by 
\begin{align*}
\check{f}([X:Y:Z]) &=[\s^{-1}XYZ:X^2Y+Y^2Z+Z^2X:XY^2+YZ^2+ZX^2],\\
\check{g}([X:Y:Z]) &=[X^3-\n^{-1}XYZ:Y^3-\n^{-1} XYZ:Z^3-\n^{-1}XYZ].
\end{align*}
\end{cor}

\begin{proof}Note that $\check{f}$ equals $g$ with $(\m,\m',\n)=(\r,\r',(3\s)^{-1})$.  Hence 
\[(\check{f}\circ f)^* \o_\r = f^*(3 \n \o_{\r'})=3\r\cdot 3\n \o_\r=3\o_\r\]
by \eqref{e4.12} and \eqref{e4.14}. 
Since the $\deg f=3$, this implies that $\check{f}$ is the dual of $f$ 
 (see \cite[III, Corollaries 5.3 and 5.6 (b)]{silverman}). 
 The same argument applies to $\check{g}$. 
\end{proof}

\begin{rmk}
Over $\C$, the formula of $\check{g}$ is  equivalent to \eqref{e2.26}, noting that $\check{g}$ corresponds via $\Psi$ to the natural surjection $\C/\angle{1,3\t}\to\C/\angle{1,\t}$. 
\end{rmk}

\section{Hypergeometric functions}\label{s5}

Here we relate Hesse elliptic curves with Gaussian hypergeometric functions over the complex numbers and over finite fields. 
In Section \ref{s5.1}, a cubic transformation formula plays a crucial role, and in Section \ref{s5.2}, we give a new proof of its finite analogue. 

\subsection{Complex hypergeometric functions}\label{s5.1}

Recall that the {\it Gaussian hypergeometric function} with parameters $a, b, c \in \C$ ($c \not\in\Z_{\le 0}$) is defined by the power series
\[F(a,b;c;x)= \sum_{n=0}^\infty \frac{(a)_n(b)_n}{(1)_n(c)_n}x^n,\]
where the Pochhammer symbol is defined by
\[(a)_n\coloneqq \frac{\G(a+n)}{\G(a)}=\prod_{i=0}^{n-1}(a+i).\] 
It converges absolutely on $|x|<1$ and is continued analytically to a multi-valued function on $\P^1-\{0,1,\infty\}$, since it is a solution of the differential equation
\[[D(D+c-1)-x(D+a)(D+b)]y=0, \quad D\coloneqq x\frac{d}{dx}.\]
Since the equation is written as
\[\left[\frac{d^2}{dx^2} +\left(\frac{c}{x}-\frac{e}{1-x}\right)\frac{d}{dx} -\frac{ab}{x(1-x)}\right]y =0, 
\quad e\coloneqq a+b+1-c, \]
its another solution is given by 
$F(a,b;e;1-x)$. 
The Schwarz map of this equation is defined by
\[\t \colon \H \to \C; \quad \t(x)= C\cdot \frac{F(a,b;e,1-x)}{F(a,b;c;x)},\]
where $C \in\C^*$ is a normalizing constant (cf. \cite[\S 5.3]{yoshida}).  
Note that the hypergeometric functions are uniquely analytically continued in the simply-connected domain $\H$. 
In our situations, theta functions give the inverse of $\t(x)$. 
In our applications below, $\t(\H)$ is a curvilinear triangle contained in $\H$. 

Among many transformation formulas for hypergeometric functions, we have Landen's formula
\begin{equation}\label{landen}
F\left(\frac{1}{2},\frac{1}{2};1; \left(\frac{1-x}{1+x}\right)^2\right)=\frac{1+x}{2} F\left(\frac{1}{2},\frac{1}{2};1;1-x^2\right).
\end{equation}
Its cubic analogue is 
\begin{equation}\label{rbb}
F\left(\frac{1}{3},\frac{2}{3}; 1; \left(\frac{1-x}{1+2x}\right)^3\right) = \frac{1+2x}{3} F\left(\frac{1}{3},\frac{2}{3}; 1; 1-x^3\right). 
\end{equation}
The last formula was known to Ramanujan, and was rediscovered and proved by Borwein--Borwein \cite[p. 694]{bb}. 
See \cite{otsubo1} for short proofs of transformation formulas including generalizations of \eqref{landen} and \eqref{rbb}. 

From these follows that, for a real number $r$ with $0<r<1$, 
\begin{align*} F\left(\frac{1}{2},\frac{1}{2};1;1-r^2\right) & =\frac{1}{\AGM_2(1,r)}, 
\\ F\left(\frac{1}{3},\frac{2}{3};1;1-r^3\right)&=\frac{1}{\AGM_3(1,r)}.
\end{align*}
The latter is \cite[Theorem 2.1]{bb}.  
We remark that if $(a,b) \mapsto (a',b')$ is the cubic AGM step and we set $x=b/a$, then \eqref{rbb} is written as 
\[a'^{-1} F\left(\frac{1}{3},\frac{2}{3}; 1; 1-\Bigl(\frac{b'}{a'}\Bigr)^3\right)=
a^{-1} F\left(\frac{1}{3},\frac{2}{3}; 1; 1-\Bigl(\frac{b}{a}\Bigr)^3\right), \]
corresponding to $\AGM_3(a',b')=\AGM_3(a,b)$. 

The following is classical. 
\begin{ppn}\label{p4.1}
If we put $x=\th_2^4(q^2)/\th_3^4(q^2)$, then we have 
\[F\left(\frac{1}{2},\frac{1}{2};1;x\right)=\th_3^2(q^2),\]
and
\[\t=i\cdot \frac{F(\frac{1}{2},\frac{1}{2};1;1-x)}{F(\frac{1}{2},\frac{1}{2};1;x)}.\]
\qed
\end{ppn}

Its cubic analogue is as follows. 

\begin{ppn}[{\cite[Theorem 2.3 (a)]{bb} and \cite[Lemma 2.9]{berndt-b-g}}]\label{p4.2}\
If we put $x=c^3(q^2)/a^3(q^2)$, then we have 
\begin{equation}\label{e5.3}
F\left(\frac{1}{3},\frac{2}{3};1;x\right)=a(q^2),
\end{equation}
and
\begin{equation}\label{e5.4}
\t= \frac{i}{\sqrt{3}} \cdot \frac{F(\frac{1}{3},\frac{2}{3};1;1-x)}{F(\frac{1}{3},\frac{2}{3};1;x)}.  
\end{equation}
\qed
\end{ppn}

We apply this result to compute the periods of Hesse elliptic curves. 
Let $\a$ (resp. $\b$) be the line segment on $\C/\langle 1,\t \rangle$ from $0$ to $1$ (resp. from $0$ to $\t$). 
Then $\{\a, \b\}$ defines a symplectic basis  of $H_1(\C/\langle 1,\t\rangle,\Z)$. 
By taking the images under $\Phi_\t$  and $\Psi_\t$, we obtain symplectic bases 
of $H_1(H(\r(q)),\Z)$ and $H_1(H(\m(q)),\Z)$, respectively.  

\begin{thm}\label{t4.3}We have 
\begin{align*}
\int_{{\Phi_\t}_*(\a)} \o_{\r(q)} &= 2\pi i \cdot \r(q)^{-1}F\left(\frac{1}{3},\frac{2}{3};1;\frac{1}{\r(q)^3}\right), 
\\ \int_{{\Phi_\t}*(\b)} \o_{\r(q)} &=-\frac{2\pi}{\sqrt{3}}\cdot \r(q)^{-1} F\left(\frac{1}{3},\frac{2}{3};1;1-\frac{1}{\r(q)^3}\right), \notag
\\\int_{{\Psi_\t}_*(\a)} \o_{\m(q)} &=-\frac{2\pi}{\sqrt{3}} \cdot \m(q)^{-1}F\left(\frac{1}{3},\frac{2}{3};1;1-\frac{1}{\m(q)^3}\right), 
\\ \int_{{\Psi_\t}_*(\b)} \o_{\m(q)} &= 2\pi i \cdot \m(q)^{-1} F\left(\frac{1}{3},\frac{2}{3};1; \frac{1}{\m(q)^3}\right).\notag
\end{align*}
\end{thm}

\begin{proof}
By Stokes' theorem and \eqref{e4.8}, we have 
\[\int_{{\Phi_\t}_*(\a)} \o_{\m(q)} = \int_0^1 \Phi_\t^*(\o_{\m(q)}) =  \int_0^12\pi i \cdot c(q^2)\, dz= 2\pi i \cdot c(q^2)
=2\pi i \cdot \frac{a(q^2)}{\r(q)},\]
and the first equation follows by \eqref{e5.3}. 
Since $\int_{{\Phi_\t}_*(\b)} \o_\m =\t \int_{{\Phi_\t}_*(\a)} \o_\m$, the second equation follows by \eqref{e5.4}. 
The remaining statements follow similarly from \eqref{e4.10}, \eqref{e5.3} and \eqref{e5.4}, noting $\r(q^\frac{1}{3})^{-3}+\m(q)^{-3}=1$. 
\end{proof}

\begin{rmk}
Mastsumoto--Terasoma--Yamazaki \cite{mty} give another basis of $H_1(H(\m),\Z)$, along which $\o_\m$ has the periods (up to constant)
\[F\left(\frac{1}{3},\frac{1}{3};\frac{2}{3};\m^3\right), \quad \m \cdot F\left(\frac{2}{3},\frac{2}{3};\frac{4}{3};\m^3\right).\]
These periods, as well as the first two periods in Theorem \ref{t4.3}, 
give fundamental solutions of the differential equation 
\[\left[D_\m(D_\m-1)- \m^3 (D_\m+1)^2\right] y=0, \quad D_\m=\m\frac{d}{d\m}.\]
This equation may also be derived as the Gauss--Manin connection of the family of Hesse curves over $\A^1-\bmm_3$, the affine modular curve $Y(\G(3))$. 
\end{rmk}

\subsection{Hypergeometric functions over finite fields}\label{s5.2} 

We recall the definition of Gaussian hypergeometric functions over finite fields as given in \cite{otsubo} (for its relation with other definitions, see \cite[Remark 2.3]{otsubo}). 
Let $\F_q$ be a finite field  with $q$ elements, and 
fix a non-trivial additive character $\psi\colon\F_q\to\C^*$, e.g. $\psi(x)=\be(\Tr_{\F_q/\F_p}(x)/p)$ where $p$ is the characteristic of $\F_q$ and $\Tr_{\F_q/\F_p}$ is the trace map. 
For a multiplicative character  $\a \colon \F_q^* \to \C^*$, the {\it Gauss sum} and its variant is defined by 
\[g(\a)=-\sum_{x \in \F_q^*}\psi(x)\a(x), \quad 
g^\0(\a)=\begin{cases} g(\a) & (\a \ne 1), \\ q & (\a =1).\end{cases}\]
These are finite analogues of the gamma function and satisfies
\[g(\a)g^\0(\a) = \a(-1)q.\]
For multiplicative characters $\a, \b, \g$, the {\it hypergeometric function} is defined by 
\[F(\a,\b;\g; \l)_{\F_q} = \frac{1}{1-q} \sum_\n \frac{(\a)_\n(\b)_\n}{(1)^\0_\n(\g)^\0_\n} \n(\l),\]
where $\n$ runs through the multiplicative characters of $\F_q$ and
\[(\a)_\n\coloneqq\frac{g(\a\n)}{g(\a)}, \quad (\a)_\n^\0\coloneqq\frac{g^\0(\a\n)}{g^\0(\a)}.\]
This is a function from $\F_q$ to $\C$ and is in fact independent of $\psi$. 
Note that $F(\a,\b;\g;0)=0$ by definition. 
Hypergeometric functions over finite fields have many properties analogous to those over the complex numbers (see for example  \cite{otsubo}). In fact, they can be regarded as different realizations of same motives (see for example \cite{asakura-otsubo}, \cite {rr}). 

Assume that $q$ is odd and let $\phi$ be the quadratic character of $\F_q^*$.  Let $L(\l)$ be the Legendre curve as defined in \ref{s3.2}. 
Then we have \cite{koike} (see also \cite[Theorem 1]{koblitz})
\[q+1-\#L(\l)(\F_q) = F(\phi,\phi;1;\l)_{\F_q}\]
for any $\l \in \F_q-\{0,1\}$. 
For the Hesse curve, we have the following point-counting result. 
\begin{ppn}\label{p7.1}
Assume that $q \equiv 1 \pmod{3}$, and fix a cubic character $\vp$ of $\F_q^*$. 
Then, we have 
\begin{equation*}
q+1 - \# H(\m)(\F_q)= F\left(\vp,\vp^2;1;\m^{-3}\right)_{\F_q}
 \end{equation*}
 for any $\m\in\F_q^*-\bmm_3$. 
 \qed
 \end{ppn}
This is \cite[Corollary 2.6]{mccarthy} when $q$ is a prime, and \cite[Theorem 3.2]{nakagawa} in general (using  \cite[Proposition 2.12]{otsubo}) (see also \cite[Theorem 2]{koblitz}). 

As a corollary, we give a short proof of the following formula, which is a finite analogue of \eqref{rbb}. 

\begin{ppn}[{Fuselier et. al. \cite[(9.1)]{fuselieretal}}]
For any $\l \in \F_q-\bmm_3$, $\l \ne -1/2$,  we have
\[F\left(\vp,\vp^2;1;\l^3\right)_{\F_q} = F\left(\vp,\vp^2;1; 1-\left(\frac{1-\l}{1+2\l}\right)^3\right)_{\!\F_q}. \]
\end{ppn}

\begin{proof}Since the case $\l=0$ is evident, assume that $\l\ne 0$ and put $\m=\l^{-1}$, so that 
\[\frac{\m+2}{\m-1}=\left(\frac{1-\l}{1+2 \l}\right)^{-1}.\] 
Recall that $q+1-\# E(\F_q)$ is the Frobenius trace for an elliptic curve $E$ over $\F_q$, 
hence is an isogeny invariant (see \cite[V, Exercise 5.4 (a)]{silverman}). 
By the isogeny $A_\m$ (see Poposition \ref{p3.6})
and Proposition \ref{p7.1}, we obtain
\[F\left(\vp,\vp^2;1;\l^3\right)_{\F_q} = F\left(\vp,\vp^2;1;\left(\frac{1-\l}{1+2\l}\right)^3\right)_{\!\F_q}.\]
On the other hand, we have by \cite[Theorem 3.15]{otsubo}
\[F(\vp,\vp^2;1; 1-x)_{\F_q}=F(\vp,\vp^2;1;x)_{\F_q} \quad (x \in \F_q-\{0,1\}),\]
noting $g(\vp)g(\vp^2)=\vp(-1)q=q$. 
Hence the proposition follows. 
\end{proof}

Recall that the Legendre elliptic curve $L(\l)$ over $\F_q$ ($p\ne 2$) is supersingular if and only if $P_L(\l)=0$, where 
\[P_L(x)\coloneqq F\left(\frac{1}{2},\frac{1}{2};1;x\right)_{\le \frac{p-1}{2}}
= \sum_{n=0}^\frac{p-1}{2} \frac{((\frac{1}{2})_n)^2}{((1)_n)^2} x^n \ \in\Z_p[x]\]
is a truncated hypergeometric series (see \cite[V, Theorem 4.1 (b)]{silverman}). 
Note that $(a)_n/(1)_n=(-1)^n \binom{-a}{n} \in \Z_p$ for any $a\in \Z_p$.

\begin{ppn}\label{p5.7}
Assume that $p \ne 3$ and $\m \in \F_q-\bmm_3$. 
\begin{enumerate}
\item The Fermat cubic $H(0)$ is supersingular if and only if $p\equiv 2\pmod{3}$. 
\item If $\m \ne 0$, then $H(\m)$ is supersingular if and only if $P_H(\m^{-3})= 0$, where 
\[P_H(x) \coloneqq F\left(\frac{1}{3},\frac{2}{3}; 1;x\right)_{\le \frac{p-1}{3}}
=\sum_{n=0}^{\lfloor \frac{p-1}{3} \rfloor} \frac{(\frac{1}{3})_n(\frac{2}{3})_n}{((1)_n)^2} x^n \ \in \Z_p[x].\]
\end{enumerate}
\end{ppn}

\begin{proof}
Put \[f(X,Y,Z)=X^3+Y^3+Z^3-3\m XYZ \in \F_q[X,Y,Z].\]
By \cite[IV, Proposition 4.21]{harts}, the supersingularity of $H(\m)$ is equivalent to the vanishing of the coefficient of $(XYZ)^{p-1}$ in $f(X,Y,Z)^{p-1}$. The coefficient is computed as 
\begin{align*}
&\sum_{3n+m=p-1} \frac{(p-1)!}{(n!)^3 m!} (-3\m)^m 
= \sum_{n=0}^{\lfloor \frac{p-1}{3} \rfloor} \frac{(p-1) \cdots (p-3n)}{(n!)^3} (-3\m)^{p-1-3n}
\\&= \sum_{n=0}^{\lfloor \frac{p-1}{3} \rfloor} \frac{(3n)!}{3^{3n}(n!)^3} \m^{p-1-3n}
=\sum_{n=0}^{\lfloor \frac{p-1}{3} \rfloor} \frac{(\frac{1}{3})_n(\frac{2}{3})_n}{((1)_n)^2} \m^{p-1-3n}, 
\end{align*}
where we used $(3n)!=3^{3n}(\frac{1}{3})_n(\frac{2}{3})_n(1)_n$, 
and the proposition follows. 
\end{proof}

\begin{rmk}
The truncated series $P_L(x)$ (resp. $P_H(x)$) can be replaced with the terminating series (i.e. the polynomial) 
\begin{align*}
F\left(\frac{1-p}{2},\frac{1-p}{2};1;x\right)
\quad (\text{resp. } \ F\left(\frac{1-p}{3},\frac{2-p}{3};1;x\right)), 
\end{align*}
since they are congruent to each other modulo $p$. 
\end{rmk}

\section{$p$-adic cubic arithmetic geometric mean}\label{s6}

\subsection{$p$-adic cubic arithmetic geometric mean}\label{s6.1}

Let $p$ be a prime number and $\C_p$ be the $p$-adic completion of an algebraic closure of the $p$-adic number field $\Q_p$. 
Let $|\phantom{p}| \colon \C_p \to \R_{\ge 0}$ be the $p$-adic absolute value, normalized so that $|p|=p^{-1}$. Put a subgroup of the unit group as
\[ U_0 =\left\{x \in \C_p \bigm| |x-1|<p^{-\frac{1}{p-1}} \right\}.\]
Then the logarithm gives a continuous isomorphism from $U_0$ to the additive group \break $\{y \in \C_p \mid |y|<p^{-\frac{1}{p-1}}\}$, 
whose inverse is given by the exponential map (see e.g. \cite[Chapter II, \S5]{neukirch}). Put further 
\[U=(U_0)^3 = \left\{x \in \C_p \bigm| |x-1|< |3| \cdot p^{-\frac{1}{p-1}} \right\}\]
($U=U_0$ unless $p=3$). 
Then we have an isomorphism $U \xrightarrow\simeq U_0$ of cubic root defined by 
\[x^\frac{1}{3}= \exp \left(\frac{1}{3} \log x\right)=\sum_{n=0}^\infty \binom{\frac{1}{3}}{n} (x-1)^n.\]

\begin{lem}\label{l5.1}
Assume that $a, b \in \C_p^*$ satisfy $\m\coloneqq a/b \in U$, and put
\[a'=\frac{a+2b}{3}, \quad b'= b \left(\frac{\m^2+\m+1}{3}\right)^\frac{1}{3}, \quad c'=\frac{a-b}{3}, \quad 
\m'=\frac{a'}{b'}=\frac{\m+2}{3 (\frac{\m^2+\m+1}{3})^\frac{1}{3}}.\]
Then we have:
\begin{enumerate}
\item $a'^3=b'^3+c'^3$. 
\item $|a|=|b|=|a'|=|b'|$. 
\item$\m' \in U_0$. 
\item $|\m'-1| = |\m-1|^3/|3|^4$.
\item If $p=3$ and $|\m-1|<|3|^2$, then $|\m'-1|<|\m-1|$. 
\end{enumerate}
\end{lem}

\begin{proof}
First, since
\[\left|\frac{\m^2+\m+1}{3}-1\right|=|\m-1| \cdot \left|1+\frac{\m-1}{3}\right|=|\m-1|,\]
we have $\frac{\m^2+\m+1}{3} \in U$ and $b'$ is well-defined. 
The statement (i) is easy to verify. 

(ii) Since $|\m|=|\frac{\m^2+\m+1}{3}|=1$, we have $|a|=|b|=|b'|$. 
Since $|a'/b|=\left|1+\frac{\m-1}{3}\right|=1$, we have $|a'|=|b|$. 

(iii) Since 
\[|\m'-1|=\left|\frac{(\m-1)+3(1-(\frac{\m^2+\m+1}{3})^\frac{1}{3})}{3(\frac{\m^2+\m+1}{3})^\frac{1}{3}}\right|=\frac{\left|(\m-1)+3(1-(\frac{\m^2+\m+1}{3})^\frac{1}{3})\right|}{|3|},\]
where $\m \in U$ and  $(\frac{\m^2+\m+1}{3})^\frac{1}{3} \in U_0$,  we have $\m' \in U_0$. 

(iv) First, we have
\[|\m'^3-1|=
\frac{|a'^3-b'^3|}{|b'^3|}= \frac{|c'^3|}{|b'^3|} =
\frac{|\frac{\m-1}{3}|^3}{|\frac{\m^2+\m+1}{3}|}=\left|\frac{\m-1}{3}\right|^3. \]
Secondly, since $\m'\in U_0$, we have 
\[|\m'^2+\m'+1|=|(\m'-1)^3+3\m'|=|3|\]
(if $p=3$, then $|(\m'-1)^3|<3^{-\frac{3}{2}}<|3|$). 
Therefore we have 
\[|\m'-1|=\frac{|\m'^3-1|}{|3|}=\frac{|\m-1|^3}{|3|^4}, \]
and the statement follows. 
The statement (v) is immediate from (iv). 
\end{proof}

\begin{ppn}\label{p5.2}
Assume that $a, b \in \C_p^*$ satisfy $\m \coloneqq a/b \in U$. 
If $p=3$, assume further that $|\m-1|<|3|^2$. 
Define sequences $a_n$, $b_n$, $\m_n$ by 
\[(a_0,b_0)=(a,b), \quad (a_{n+1},b_{n+1})=(a_n',b_n'), \quad \m_n=a_n/b_n\]
(well-defined by Lemma \ref{l5.1}). 
Then $\m_n$ converges {\it cubically} to $1$, 
and both $a_n$ and $b_n$ converge cubically to the same element. 
\end{ppn}
\begin{proof}First, the sequence $\m_n$ converges cubically to $1$ by Lemma \ref{l5.1} (iv). 
Since 
\[|a_{n+1}-a_n|=\frac{|a_n-b_n|}{|3|}=\frac{|b_n|\cdot|\m_n-1|}{|3|}=\frac{|a_0|\cdot|\m_n-1|}{|3|},\]
by Lemma \ref{l5.1} (ii), the sequence $a_n$ converges. 
Since $|\m_n-1|$ decreases monotonously by Lemma \ref{l5.1} (iv) and (v), we have
$|\lim_{n\to\infty}\limits a_n - a_n|=|a_{n+1}-a_n|$, and   $a_n$ converges cubically. 
The remaining assertions follow easily. 
\end{proof}

\begin{dfn}Under the assumption of Proposition \ref{p5.2}, we define
\[\AGM_3(a,b) =\lim_{n\to\infty}\limits a_n=\lim_{n\to\infty}\limits b_n, \]
and call it the ($p$-adic) {\it cubic arithmetic-geometric mean} of $a$ and $b$. 
Note that if $a$ and $b$ are contained in a complete subfield $K\subset \C_p$, then $\AGM_3(a,b) \in K$. 
\end{dfn}

If $p=3$ and $|\m-1|=|3|^2$, then 
$|\m_n-1|=|3|^2$ for all $n$ by Lemma \ref{l5.1} (v). 
Hence $\m_n$ does not converge to $1$, but see Theorem \ref{t7.1} below.  

\subsection{Reduction of Hesse elliptic curves}\label{s6.2}

Let $K$ be a finite extension of $\Q_p$ contained in $\C_p$, $\sO_K$ be its integer ring, $\pi$ be a uniformizer and $k=\sO_K/(\pi)$ be the residue field. 
Let $e$ be the ramification index of $K$ over $\Q_p$, i.e. $|\pi|=|p|^{1/e}$. 
Let $\tilde x \in k$ denote the residue class of $x\in \sO_K$. 

For $\m\in K-\bmm_3$, let $H(\m)$ be the Hesse elliptic curve over $K$ as defined in \S \ref{s4.1}. 

\begin{ppn}\label{p6.6} 
Suppose that $p \ne 3$. 
\begin{enumerate}
\item If $|\m^3-1|=1$, then $H(\m)$ has good reduction, and the reduction is supersingular if and only if one of the following holds:
\begin{enumerate}
\item $|\m|<1$ and $p\equiv 2 \pmod{3}$. 
\item $|\m|=1$ and
\[P(\m^{-3}) \equiv 0 \pmod{\pi},\] where $P(x) \in \Z_p[x]$ is as in Proposition \ref{p5.7}. 
\end{enumerate}
\item If $|\m^3-1|>1$, then $H(\m)$ has split multiplicative reduction. 
\item If $|\m^3-1|<1$, then $H(\m)$ has multiplicative reduction, which splits 
if and only if $\o \in K$.  
\end{enumerate}
\end{ppn}

\begin{proof}(i) Since $\m^3 \not\equiv 1 \pmod{\pi}$, the reduction is given by the same equation \eqref{e.hesse}, and the result follows from Proposition \ref{p5.7}. 

(ii) The reduction is given by $XYZ=0$, which has only three $k$-rational nodes. 

(iii) The reduction is given by 
\[(X+Y+Z)(X+\tilde \o Y+\tilde\o^2Z)(X+\tilde\o^2Y+\tilde\o Z)=0\]
where $\tilde\o\in k$ is the residue class of $\o$. It has again three nodes, which are $k(\tilde\o)$-rational. 
Since  $\tilde\o \in k$ if and only if $\o\in K$, we obtain the result. 
\end{proof}

\begin{rmk}When $|\m^3-1|<1$, let $n$ be the positive integer such that $|\m^3-1|=|\pi|^n$. 
Then, the special fiber of a regular integral model of $H(\m)$ is $I_{3n}$. 
This follows either from a concrete description of the model, 
or from Tate's algorithm using a minimal Weierstrass equation of $H(\m)$. 
\end{rmk}

\begin{ppn}\label{p6.8}
Suppose that $p=3$. Then, $H(\m)$ has potentially good reduction if and only if 
$|3|^3\le|\m^3-1|\le |3|^{-3}$, in which case the reduction is ordinary if and only if $|\m^3-1|=|3|^{\pm 3}$. 
\end{ppn}

\begin{proof}
Since an elliptic curve over $K$ has potentially good reduction if and only if its $j$-invariant is integral (cf. \cite[Appendix A, Corollary 1.4]{silverman}), the first statement follows from \eqref{e4.j}. 
In characteristic $3$, there is a unique supersingular elliptic curve (up to isomorphism) given by $y^2=x^3+x$ (see \cite[V, Theorem 4.1(c) and Example 4.5]{silverman}), and its $j$-invariant is $0$ ($=1728$). 
Since the $j$-invariant of the reduction is the reduction of the $j$-invariant, 
the reduction of $H(\m)$ is ordinary if and only if $|j(H(\m))|=1$. 
Hence the last statement follows from \eqref{e4.j}. 
\end{proof}

For later use, let us give an explicit form of the reduction under the assumption $\o\in K$. 
If $\m$ is not integral, then $\a(\m)=(1+2\m^{-1})(1-\m^{-1})^{-1}$ is integral, and 
\[|\a(\m)^3-1|=\frac{|9\m^{-1}(1+\m^{-1}+\m^{-2})|}{|1-\m^{-1}|^3}
=\frac{|3|^2}{|\m|}=\frac{|3|^2}{|\m^3-1|^\frac{1}{3}}.\]
In particular if $|\m^3-1|=|3|^{-3}$, then $|\a(\m)^3-1|=|3|^3$ and since $H(\m) \simeq H(\a(\m))$, 
we are reduced to the case $|\m^3-1|=|3|^2$. 
Then, we have 
$|\m-\o^i|=|3|^2$ for some $i\in\Z/3\Z$ and $|\m-\o^j|=|3|^\frac{1}{2}$ for $j \ne i$. 
Since $H(\m) \simeq H(\b(\m))$, it suffices to consider the case $|\m-1|=|3|^2$.

\begin{ppn}\label{p6.9}
Suppose that $p=3$ and $\o \in K$. For $\m \in K-\bmm_3$ with $|\m-1|=|3|^2$, put $\m=1+9m$. Then the reduction of $H(\m)$ is given by
\[E(\tilde{m}) : y^2z = x^3+x^2z-\tilde{m}xz^2+\tilde{m}^2z^3,\]
with the $j$-invariant $j(E(\tilde{m}))=\tilde{m}^{-3}$. We have 
\[E(\tilde m)[3] =\{O, [0: \pm \tilde{m}: 1]\},\]
and  
\begin{align*} 
\Ker(H(\m)[3] \to E(\tilde m)[3]) &= \{X+Y+Z=0\} \\&= \{[0:-1:1], [-1:1:0], [1:0:-1]\}.\end{align*}
\end{ppn}

\begin{proof}
Replacing  $[x:y:z]$ with $[(1-\o)^2(x-z):(1-\o)^3(y+z):z]$, the Weierstrass form \eqref{weier1} becomes 
\[E : y^2z+\frac{3+9m}{1-\o} xyz+\frac{(9m)^2}{3 (1-\o)^3} yz^2 = x^3-\frac{3}{(1-\o)^2} x^2z-\frac{9m}{(1-\o)^4}xz^2-\frac{(9m)^2}{3(1-\o)^6}z^3.\]
It is a minimal Weierstrass equation with unit discriminant $((\m^3-1)/3^3)^3$. 
Computing the residue classes of the coefficients, we obtain the equation of $E(\tilde m)$. 
By \eqref{e4.j}, we have 
\[j(H(\m))=\left(\frac{(1+3m)(1+9m)(1+27m^2)}{m(1+9m+27m^2)}\right)^3 \equiv m^{-3} \pmod{3}, \]
and the $j$-invariant of $E(\tilde{m})$ follows. 
The isomorphism $H(\m) \to E$ is given by 
\begin{align*}
[x:y:z] = &
\left[ 
\frac{-M(X+Y+Z)+(\m X+Y+Z)}{(1-\o)^2}: \right.
\\
& \quad \left. \frac{M(X+Y+\m Z)-(\m X+Y+Z)}{(1-\o)^3}: 
\m X+Y+Z
\right], 
\end{align*}
where we put $M=\frac{\m^2+\m+1}{3}=1+9m+27m^2$. 
The image of $H(\m)[3]$ is easily computed and the remaining statements follow. 
\end{proof}

\subsection{Canonical lift of Hesse elliptic curves}

In this section, we suppose that $p=3$. 
Let $K$ be a finite extension of $\Q_3$ containing $\o$, of residual degree $d$ and $K_0 \subset K$ be the maximal unramified subfield. 
Let $\s \in \Gal(K_0/\Q_3)$ be the lift of $3$-Frobenius. Put $q=3^d$. 

For an ordinary elliptic curve $E/k$, there exists a lift $E^\can/K_0$ such that $\End(E^\can)\simeq  \End(E)$. 
This $E^\can$ is unique up to isomorphism and is called the {\it canonical lift} of Serre--Tate \cite{lst} (cf. \cite[Appendix]{messing}). 
For ordinary elliptic curves $E_1, E_2/k$, we have 
\[\Hom(E_1^\can,E_2^\can) \simeq \Hom(E_1,E_2).\]
In particular, the $q$-th power Frobenius endomorphism $\Fr_q \in \End(E)$ lifts uniquely to $\vp_q \in \End(E^\can)$. 

We consider $H(\m)$ with ordinary good reduction. 
By the argument after Proposition \ref{p6.8}, we may assume that $|\m-1|=|3|^2$. 

The main result of this section is the following cubic analogue of \cite[Theorem 1.1, 1]{km}. 
\begin{thm}\label{t7.1}
Assume that $\m \in K$ satisfies $|\m-1|=|3|^2$, and let $\m_n$ be the sequence defined as in  Proposition \ref{p5.2}. Then the limit 
\[\m_\infty\coloneqq \lim_{n \to \infty}\limits \m_{dn}\]
exists in $K_0(\o)$. The element $\m_\infty$ is characterized by the following properties: 
\begin{enumerate}
\item $|\m_\infty-1|=|3|^2$, 
\item $\m_\infty'=\s\m_\infty$, 
\item $H(\m_\infty) \simeq E(\tilde{m})^\can \ot_{K_0} K_0(\o)$, where $\m=1+9m$.  
\end{enumerate}
\end{thm}

\begin{rmk}
By the proof of the theorem and Lemma \ref{l7.4} below, the convergence of $\m_{dn}$ is linear. 
\end{rmk}

To prove the theorem, we prepare three lemmas. 
Recall the group $G=\langle \a, \b \rangle$ and the element $\g=\b\a\b \in G$ from Section \ref{s4.1}.

\begin{lem}\label{l7.2} 
Assume that $|\m-1|=|3|^2$. 
\begin{enumerate}
\item For $g\in G$, we have
\[|g(\m)-1|= \begin{cases}
|3|^2 & \text{if $g=\g^i$ for some $i\in\Z/3\Z$}, \\
|3|^{-1} & \text{if $g=\b^i\a$ for some $i \in \Z/3\Z$}, \\
|3|^\frac{1}{2} & \text{otherwise}. 
\end{cases}\]
\item For $i=1, 2$, we have $|\g^i(\m)-\m|=|3|^\frac{5}{2}$. 
\end{enumerate}
\end{lem}

\begin{proof}
(i) First, the case $g=1$ is evident and if $i \ne 0$, then 
\[|\g^i(\m)-1|=\left|\o^i\frac{\o^i\m+2}{\o^i\m-1}-1\right|
=\frac{|1-\o^i| |\m-1|}{|\o^i(\m-1)-(1-\o^i)|}=|\m-1|=|3|^2. \]
Secondly, for any $i$, 
\[|\b^i\a(\m)-1|=\left| \o^i \frac{\m+2}{\m-1} \right| =\frac{|\o^i(\m-1)+3\o^i|}{|\m-1|}=\frac{|3|}{|\m-1|}=|3|^{-1}.\]
Thirdly, for $i \ne 0$,  
\[|\b^i(\m)-1|=|\o^i(\m-1)+(\o^i-1)|=|\o^i-1|=|3|^\frac{1}{2}.\] 
Fourthly, for $j \ne 0$, 
\[|\a\b^j(\m)-1|=\left|\frac{\o^j\m+2}{\o^j\m-1}-1\right|
=\frac{|3|}{|\o^j(\m-1)-(1-\o^j)|}=\frac{|3|}{|1-\o^j|}=|3|^\frac{1}{2}.\]
Finally, if $i =-j \ne 0$, then 
\[|\b^i\a\b^j(\m)-1|=\frac{|(1-\o^j)(\m-1)-3\o^j|}{|\o^j\m-1|}=\frac{|3|}{|1-\o^j|}=|3|^\frac{1}{2}.\]

(ii) We have 
\[|\g^i(\m)-\m|=\frac{\left|-\o^i\left(3(\m-1)+(\m-1)^2\right)\right|}{|\o^i\m-1|}=\frac{|3(\m-1)|}{|1-\o^i|}=|3|^\frac{5}{2}.\]
\end{proof}

By Lemma \ref{l5.1} (iii), the assumption $|\m-1|=|3|^2$ implies $|\m'-1|=|3|^2$. 
\begin{lem}\label{l7.3}
Assume that $|\m-1|=|\n-1|=|3|^2$. 
\begin{enumerate}
\item We have $\m'=\n'$ if and only if 
$\n = \g^i(\m)$ 
for some $i \in \Z/3\Z$. 
\item The reductions of $H(\m)$ and $H(\n)$ are isomorphic to each other if and only if $|\m-\n| < |3|^2$. 
\item If $\m=1+9m$ and $\m'=1+9m'$, then we have $m'=m^3 \pmod{3}$. 
\end{enumerate}
\end{lem}

\begin{proof}(i) The condition $\m'=\n'$ is written as 
\[(\m^2+\m+1)(\n+2)^3=(\m+2)^3(\n^2+\n+1).\]
This is a cubic equation in $\n$ with one solution $\n=\m$, and it is easy to verify that $\g(\m)$ and $\g^2(\m)$ are the other solutions. 
Note that these three are distinct by the proof of Proposition \ref{p4.3}. 

(ii) This follows from  Proposition \ref{p6.9}. 

(iii) Since $\m$ and $\m'$ are related by 
$\frac{\m^2+\m+1}{3} \cdot \m'^3 = (\frac{\m+2}{3})^3, 
$
we have 
\[(1+9m+27m^2)(m'+9m'^2+27m'^3)=m^3, \]
and the assertion follows. 
\end{proof}

\begin{lem}\label{l7.4}
If $|\m-1|=|\n-1|=|3|^2$, then we have 
\[|\m'-\n'| \le \max\left\{|3|\cdot |\m-\n|, |3|^{-4} \cdot |\m-\n|^3 \right\}. \]
In particular, 
\begin{enumerate}
\item $|\m'-\n'| \le |\m-\n|$, 
\item $|\m-\n|<|3|^2$ implies $|\m'-\n'|<|\m-\n|$, 
\item $|\m-\n| \le |3|^\frac{5}{2}$ implies $|\m'-\n'| \le |3|\cdot |\m-\n|$. 
\end{enumerate}
\end{lem}

\begin{proof}

By the proof of Lemma \ref{l7.3} (iii), we have $P(m,m')=0$ where
\[P(x,y)\coloneqq (1+9x+27x^2)(y+9y^2+27y^3)-x^3.\] 
Consider the Taylor expansion of $P(x,y)$ at $(x,y)=(m,m')$: 
\[P(x,y)= \sum_{0\le i,j \le 3} a_{i,j} (x-m)^i(y-m')^j, \quad 
a_{i,j}\coloneqq \frac{\frac{\pd^{i+j}}{\pd x^i \pd y^j} P(m,m')}{i!j!}. \]
Then one computes easily
\[a_{0,0}=0, \quad |a_{1,0}|=|3|, \quad |a_{0,1}|=1, \quad a_{3,0}=-1,\]
and $|a_{i,j}|\le|3|$ for the remaining indices $(i,j)$. 
Writing $\n=1+9n$ and $\n'=1+9n'$, we have $P(n,n')=0$. 
This is written as
\[\left(a_{0,1}+a_{0,2}(n'-m')+a_{0,3}(n'-m')^2\right)(n'-m')=-\sum_{i\ge 1, j \ge 0} a_{i,j} (n-m)^i(n'-m')^j,\]
and it implies $|n'-m'|\le |n-m|$. Then, since 
\[a_{0,1}(n'-m')=(n-m)^3-\sum_{(i,j) \ne (0,0), (0,1), (3,0)} a_{i,j} (n-m)^i(n'-m')^j,\]
we obtain $|n'-m'| \le \max\{|n-m|^3, |3(n-m)|\}$, and the proposition follows.  
\end{proof}

\begin{proof}[Proof of Theorem \ref{t7.1}]
Put $K_1=K_0(\o)$. 
By Proposition \ref{p6.9}, all the $3$-torsion points of $E(\tilde{m})$ are defined over $k$.  
It follows that 
$E(\tilde{m})^\can[3] \simeq \bmm_3 \times \Z/3\Z$ 
as Galois modules (see \cite[V, \S3 and Appendix]{messing}), and 
all the $3$-torsion points of $E(\tilde{m})^\can$ are defined over $K_1$. 
Therefore, we have  
\[H(\n) \simeq E(\tilde{m})^\can \ot_{K_0} K_1\]
for some $\n \in K_1$ (see Remark \ref{r4.1}). Since $H(\n)$ has ordinary good reduction, we can choose $\n$ satisfying $|\n-1|=|3|^2$ by Proposition \ref{p6.9} and the discussion after that. 
This choice is unique up to the action of $\g$ by Lemma \ref{l7.2}

By the functoriality of canonical lifts, we have an isomorphism 
\[H(\s\n) \simeq E(\tilde{m}^3)^\can \ot_{K_0} K_1\] 
and a degree-$3$ isogeny $h\colon H(\n) \to H(\s\n)$ 
which makes the diagram
\[\xymatrix{ 
H(\n)(K_1)  \ar[r]^{h} \ar[d]_{\Red_\n} & H(\s\n)(K_1) \ar[d]_{\Red_{\s\n}}
\\E(\tilde{m})(k)  \ar[r]^{\Fr_3} & E(\tilde{m}^3)(k)
}\]
commute, where the vertical maps are the reductions. 
Since $h$ is of degree $3$, we have $\Ker(h) \subset H(\n)[3]$. 
Since $\Fr_3$ is an isomorphism on $3$-torsion, we have 
\[\Ker(h) \subset  \Ker(\Red_{\s\n}\circ h) \subset \Ker(\Red_\n|H(\n)[3]) = \{X+Y+Z=0\}\]
by Proposition \ref{p6.9}. 
Comparing the cardinalities, we obtain $\Ker(h)=\{X+Y+Z=0\}$. 
Hence by Theorem \ref{alg-isogeny2}, we have $H(\s\n) \simeq H(\n')$. 
By Proposition \ref{p4.3} and Lemma \ref{l7.2} (i),  we have $\n'=\g^i(\s\n)$ for some $i$. 
By Lemma \ref{l7.3} (i), 
\[(\g^i(\n))'=\n'=\g^i(\s\n)=\s \g^i(\n).\] 
Therefore, there is a unique $\n$ as above satisfying $\n'=\s\n$. 
Then we have $\n_n=\s^n \n$ by induction as: 
\[\n_{n+1}\coloneqq \n_n'=(\s^n\n)'=\s^n(\n')=\s^n(\s\n)=\s^{n+1}\n.\]

Now, since $H(\m)$ and $H(\n)$ have the same reduction, we have $|\m-\n|<|3|^2$ by Lemma \ref{l7.3}  (ii). 
Therefore, $|\m_n-\n_n|\to 0$ by Proposition \ref{l7.4} (ii) (and that $K$ is discretely valued). 
Since $\s^d=1$, we have $\n_{dn}=\n$ for all $n$, and hence $\m_{dn} \to \n$. 
Thus the theorem is proved. 
\end{proof}

\section*{Acknowledgements}
I would like to thank Takumi Okada for useful discussions and for allowing me to include his results in this paper. The author is supported by JSPS KAKENHI Grant JP24K06682.


\end{document}